\documentclass[11pt]{amsart} 

\usepackage{amsmath,amssymb,amsthm, verbatim}

\newtheorem{theorem}{Theorem}[section]
\newtheorem{lemma}[theorem]{Lemma}
\newtheorem{proposition}[theorem]{Proposition}

\newtheorem{corollary}[theorem]{Corollary}

\newtheorem*{theorema}{Theorem A}
\newtheorem*{theoremb}{Theorem B}

\newtheorem*{hypothesis*}{Hypothesis}
\newtheorem*{hypothesisone}{Hypothesis 1}
\newtheorem*{lemmaa}{Lemma A}
\newtheorem*{lemmab}{Lemma B}

\newtheorem{hypothesis}{Hypothesis}
\newtheorem*{ihypothesis}{Inductive Hypothesis}

\newcommand{\Galph}{G_\alpha}

\newcommand{\curlyc}[1]{\mathcal{C}_#1}
\newcommand{\spaceP}{\mathcal{P}}

\newcommand{\Fix}{{\rm Fix}}
\newcommand{\curly}{\mathcal{C}}
\newcommand{\Fiel}{\mathbb{F}}
\newcommand{\Fq}{\mathbb{F}_q}
\newcommand{\Fqr}{\mathbb{F}_{q^r}}

\newcommand{\GL}{\mathrm{GL}}

\newcommand{\Aut}{\mathrm{Aut}}
\newcommand{\PGL}{\mathrm{PGL}}
\newcommand{\SL}{\mathrm{SL}}
\newcommand{\PSL}{\mathrm{PSL}}

\newcommand{\SU}{\mathrm{SU}}
\newcommand{\PSU}{\mathrm{PSU}}

\newcommand{\GO}{\mathrm{GO}}

\newcommand{\POmega}{\mathrm{P\Omega}}

\newcommand{\Sp}{\mathrm{Sp}}
\newcommand{\PSp}{\mathrm{PSp}}

\begin{document}

\title{Transitive projective planes and insoluble groups}
\author{Nick Gill}

\address{Department of Mathematics, The Open University, Walton Hall, Milton Keynes, MK7 6AA, UK}
\email{n.gill@open.ac.uk}

\begin{abstract}
Suppose that a group $G$ acts transitively on the points of $\mathcal{P}$, a finite non-Desarguesian projective plane. We prove that if $G$ is insoluble then $G/O(G)$ is isomorphic to $\SL_2(5)$ or $\SL_2(5).2$.

\end{abstract}

\maketitle

\section{Introduction}

In 1959, Ostrom and Wagner proved that if a finite projective plane, $\spaceP$, admits an automorphism group which acts 2-transitively on the set of points of $\spaceP$, then $\spaceP$ is Desarguesian \cite{oswa}. Ever since this result appeared it has been conjectured that the same conclusion holds if the phrase {\it 2-transitively} is replaced by {\it transitively}.

A number of results have appeared which partially prove this conjecture under certain extra conditions. Most notably, in 1987 Kantor made use of the Classification of Finite Simple Groups to prove that if $\spaceP$ has order $x$ and $\spaceP$ admits a group $G$ which acts {\it primitively} on the set of points of $\spaceP$, then either $\spaceP$ is Desarguesian and $G\geq \PSL(3,x)$, or else $x^2+x+1$ is a prime and $G$ is a regular or Frobenius group of order dividing $(x^2+x+1)(x+1)$ or $(x^2+x+1)x$ \cite{kantor}.

In the 25 years since Kantor's result appeared, a number of authors have studied what remains, with particular attention being given to the special case where $G$ acts {\it flag-transitively} and so has order $(x^2+x+1)(x+1)$. Even in this special case, though, the number theory involved is very delicate and the question of whether a flag-transitive non-Desarguesian finite projective plane exists is still wide open (see, for instance, \cite{thas, thaszagier}).

In this paper we present two main results which depend only on the supposition that a group acts transitively on the set of points of a non-Desarguesian plane. These results build on those given in \cite{gill2} and constitute the closest approach to a proof of the conjecture so far. Indeed, in a sense that will be made clear below, the results presented here do for point-transitive finite projective planes what Kantor's result did for point-primitive finite projective planes. In light of the difficulties encountered in trying to strengthen Kantor's original result, it is fair to say that the results presented here are the best one could hope for without significant developments in the associated number theory.

In order to state the two results we make some definitions pertaining to a group $G$. We define $O(G)$ to be the largest odd-order normal subgroup in $G$ and $F(G)$ to be the Fitting subgroup of $G$.

\begin{theorema}
Suppose that a group $G$ acts transitively on the set of points of $\spaceP,$ a finite non-Desarguesian projective plane. Then the Sylow $2$-subgroups of $G$ are cyclic or generalized quaternion.
\end{theorema}

The statement of the next theorem is quite long and technical. The corollaries that come after it will make clear the power of the statement.

\begin{theoremb}
Suppose that a group $G$ acts transitively on the set of points of a finite non-Desarguesian projective plane of order $x$. Then $x=u^2$ for some integer $u$ and one of the following holds:
\begin{enumerate}
 \item $G$ has a normal $2$-complement (and so $G$ is soluble).
\item $G/O(G)$ is isomorphic to $\SL_2(3)$ or to a non-split degree $2$ extension of $\SL_2(3)$ (and so $G$ is soluble).
\item $G$ has a subgroup of index at most $2$ equal to $O(G)\rtimes K$ where $K$ is a group isomorphic to $\SL_2(5)$. Furthermore if $N_t$ is a Sylow $t$-subgroup of $F(G)$, for some prime $t$, then one of the following holds:
\begin{enumerate}
 \item[(i)] $t$ divides $u^2+u+1$;
 \item[(ii)] $t$ divides $u^2-u+1$, $N_G(N_t)$ contains a subgroup $H$ such that $H\cong SL_2(5)$, $H$ fixes a point  of $\spaceP$,  $N_t\rtimes H$ is a Frobenius group, and $N_t$ is abelian.
\end{enumerate}
Furthermore there exists a prime $t$ dividing both $|F(G)|$ and $u^2-u+1$.
\end{enumerate}
\end{theoremb}

The first corollary is one of the results mentioned in the abstract. It requires no proof.

\begin{corollary}\label{c: main}
Suppose that a group $G$ acts transitively on the set of points of $\spaceP,$ a finite non-Desarguesian projective plane. If $G$ is insoluble then $G/O(G)$ contains a subgroup $H$ of index at most $2$ such that $H$ is isomorphic to $SL(2,5)$.
\end{corollary}

Appealing to the Odd Order Theorem, we observe that Corollary~\ref{c: main} implies the following: If $G$ acts transitively on the points of a finite non-Desarguesian projective plane, then a composition factor of $G$ is either of prime order or is isomorphic to $A_5$. In other words we are very close to demonstrating that any counterexample to the conjecture mentioned above must be soluble. Such a conclusion is clearly the best one could hope for by making direct use of the Classification of Finite Simple Groups. Thus Corollary~\ref{c: main} brings the study of point-transitive finite projective planes to the same stage that Kantor's 1987 result brought the study of point-primitive finite projective planes \cite{kantor}. What is left is likely best attacked using number theory, and is undoubtedly very difficult indeed.

The next corollary is also reminiscent of Kantor's result, but in a different way. It is well-known that a finite Desarguesian projective plane of order $q$ admits an odd order group of automorphisms acting transitively on the set of points (take a {\it Singer cycle} in $\PGL_3(q)$). The next result asserts that the same is true for general finite projective planes.

\begin{corollary}\label{c: odd order}
Suppose $\spaceP$ is a finite projective plane admitting an automorphism group acting transitively on the set of points of $\spaceP$. Then there is a group of odd order that acts transitively on the set of points of $\spaceP$.
\end{corollary}

\subsection{Methods and structure}

Our proof of Theorem A is group-theoretic. Our main tool for the proof is the following result which is of interest in its own right.

\begin{lemmaa}\label{la}
Suppose that $n$ is a positive integer and that $q=p^a$ for some integer $a$ and prime $p\geq 7$. Let $H$ be an even order subgroup of $\GL_n(q)$. Then there exists an involution $g\in H$ such that
$$|H:C_H(g)|_{p'}\leq q^{n-1}+\dots+q+1.$$
\end{lemmaa}

(For integers $k,\ell$, we write $k_{\ell'}$ for the largest integer dividing $k$ that is coprime to $\ell$.) Note that it is quite possible that the bound given in Lemma A also holds when $p=3$ and/or $5$, however we have not attemped to prove this.

The connection between Lemma A and Theorem A is probably not immediately obvious. This connection is explored in \S\ref{section:framework}, the culmination of which is Proposition~\ref{prop:implication} and the subsequent corollary which asserts that Lemma A implies Theorem A.

In \S\ref{S: progress} we give a proof of Lemma A. This proof involves a detailed examination of subgroups of $\GL_n(q)$, and the involutions that they contain. In particular there is no reference to groups acting on projective planes in \S\ref{S: progress}. 

In addition to Lemma A, there are several results in \S\ref{S: progress} that may be of independent interest. In particular one of the results that we prove (see Lemmas~\ref{l: cn} and \ref{l: oddgln}) bears reproduction here:

\begin{proposition}\label{p: oddgln}
Suppose $H$ is an odd order subgroup of $\GL_n(q)$ where $q=p^f$ and $p\geq 7$. Then 
 \[|H|_{p'}< \frac 12 3^{n/3} q^n.\]
\end{proposition}

This result should be compared with, for instance, \cite[Theorem A]{collins}. The bound given in Proposition~\ref{p: oddgln} is not too far away from being sharp: suppose, for convenience, that $q\equiv 3\pmod 4$ and $n\equiv 3\pmod 6$, then one can easily find an odd-order subgroup in $\GL_n(q)$ of order $n\left((\frac32(q^3-1)\right)^{n/3}$, which is roughly a factor of $\frac{n}{2^{n-1}}$ away from the stated bound. 

The methods used to prove Proposition~\ref{p: oddgln} could, in principle, be applied to give much sharper bounds. Indeed the methods could also be applied to give bounds for $|H|_{p'}$ under the weaker assumption that $H$ is soluble (rather than odd-order). 

Finally in \S\ref{S: quaternionsylowtwos} we analyse the situation where $G$ is insoluble, and we prove Theorem B and Corollary~\ref{c: odd order}. 

Note that the methods used in different sections vary considerably and hence so does our notation. We explain our notation at the start of each section or subsection.

\subsection{Acknowledgments}

Thanks are due first to Professor Jan Saxl, who suggested to me some years ago that I should be able to prove a result similar to Theorem B by building on methods from my paper \cite{gill2}. It has taken me a long time to prove him right. 

I would also like to thank Professor Cheryl Praeger and an anonymous referee for help with the exposition of this manuscript, and Professors Marty Isaacs, L\' aszl\' o Pyber and Trevor Wooley who gave technical advice at crucial junctures.

\section{A framework to prove Theorem A}\label{section:framework}

Our aim in this section is to set up a framework to prove Theorem A. In order to do this we will split into two subsections. The first subsection outlines some basic group theory results that will be needed in the remainder of the paper. In the second subsection we will apply these results to the projective plane situation; in particular we will state Lemma A, and will demonstrate that Lemma A implies Theorem A.

\subsection{Some background group theory}\label{S: background}

Throughout this subsection we use standard group theory notation. For an element $g\in G,$ we write $g^G$ for the set of $G$-conjugates of $g$ in $G$. A cyclic group of order $n$ will sometimes just be written $n$. We write $G=N.H$ for an extension of $N$ by $H$; in other words $G$ contains a normal subgroup $N$ such that $G/N\cong H$. An element $g$ is an {\it involution} if $g^2=1$ and $g\neq 1$.

\begin{lemma}\label{L: oddnormal}
Let $G,H,N$ be groups with $N\lhd H \leq G$, and let $g\in N_G(H)$. If $|N|$ is odd then
$$|H:C_H(g)|=|N:C_N(g)|\times|H/N: C_{H/N}(gN)|.$$
\end{lemma}
\begin{proof}
Take $C\leq H$ such that $C/N= C_{H/N}(gN)$. Then $C\geq C_H(g)$. Let $N^*=\langle g,N \rangle\cong N.2$ and take $c\in C$. Then $g^c\in N^*$.

Since $|N|$ is odd this implies that $g^{cn}=g$ for some $n\in N$ by Sylow's theorem. Thus $C=N.C_H(g)$. Then
$$|H:C|=|H:N.C_H(g)|=\frac{|H:C_H(g)|}{|N:N\cap C_H(g)|}=\frac{|H:C_H(g)|}{|N:C_N(g)|}.$$
Since $|H:C|=|H/N: C_{H/N}(gN)|$ we are done.
\end{proof}

Note that the lemma also applies with $g\in H=G$, and we will usually use it in this context.

\begin{lemma}\label{L: sylowtwos}
Let $H$ be a group and let $N\lhd H$. Let $t$ be a prime, let $P$ be a Sylow $t$-subgroup of $N$, and let $g\in P$. Then
$$\frac{|g^H|}{|g^N|}=\frac{|g^H\cap P|}{|g^N\cap P|}.$$
\end{lemma}
\begin{proof}
Observe that, by the Frattini argument, the set of $N$-conjugates of $P$ is the same as the set of $H$-conjugates of $P$; say there are $c$ of these. Let $d$ be the number of such $N$-conjugates of $P$ which contain the element $g$. Now count the size of the following set in two different ways:
$$|\{(x,Q):x\in g^H, Q\in P^H, x\in Q\}| = |g^H|d = c|g^H\cap P|.$$
Similarly we count the size of the following set in two different ways:
$$|\{(x,Q):x\in g^N, Q\in P^N, x\in Q\}| = |g^N|d = c|g^N\cap P|.$$
Our result follows.
\end{proof}

We now use the previous two results to study the centralizer of an involution in a subgroup of a direct product of groups.

\begin{lemma}\label{L: invcentralizer}
Let $H_1,\dots, H_r$ be groups, let $H$ be a subgroup of $H_1\times\dots\times H_r$, and let $g$ be an involution in $H$. For $i=1,\dots, r$, define
\begin{itemize}
 \item $L_i$ to be the projection of $H$ to $H_i\times H_{i+1}\times\dots\times H_r$,
 \item $g_i$ to be the projection of $g$ to $H_i\times H_{i+1}\times\dots\times H_r$.
\end{itemize}
For $i=1,\dots, r-1$, define
\begin{itemize}
 \item $\psi:L_i\to L_{i+1}$ to be the natural projection, and
 \item $T_i$ to equal $\ker(\psi_i)$.
\end{itemize}
Let $k$ be the smallest integer such that $|T_i|$ is even, and let $P$ be a Sylow $2$-subgroup of $T_k$. 
\begin{enumerate}
 \item If $k=r$, then
 \[|H:C_H(g)|=\left(\prod_{i=1}^{r-1}|T_i:C_{T_i}(g_i)|\right)|L_r:C_{L_r}(g_r).\]
\item If $k<r$, then,  provided $g_{k+1}=\cdots = g_r=1$,
 \end{enumerate}
\[|H:C_H(g)|=\left(\prod_{i=1}^{k}|T_i:C_{T_i}(g_i)|\right)\frac{|(g_k)^{L_k}\cap P|}{|(g_k)^{T_k}\cap P|}.\]
\end{lemma}
\begin{proof}
If $T_1$ has odd order then Lemma \ref{L: oddnormal} implies that
\begin{eqnarray*}
|H:C_H(g)|&=&|T_1:C_{T_1}(g)|\times |H/T_1: C_{H/T_1}(T_1g)| \\
&=& |T_1:C_{T_1}(g)|\times |\psi_1(H): C_{\psi_1(H)}(\psi_1(g))| \\
&=& |T_1:C_{T_1}(g_1)|\times |L_2: C_{L_2}(g_2)|.
\end{eqnarray*}
Now $L_2$ is a subgroup of $H_2\times\dots\times H_r$ and so we can iterate the procedure. This implies that
$$|H:C_H(g)|=\left(\prod_{i=1}^{k-1}|T_i:C_{T_i}(g_i)|\right)|L_k:C_{L_k}(g_k)|.$$
If $k=r$ then we are done. If $k<r$ then we must calculate the centralizer of $g_{k}$ in $L_{k}\leq H_k\times\dots\times H_r$. Then we apply Lemma \ref{L: sylowtwos} using $T_k$ for our normal subgroup $N$. Then
$$|L_k:C_{L_k}(g_k)|=|T_k:C_{T_k}(g_k)|\times \frac{|(g_k)^{L_k}\cap P|}{|(g_k)^{T_k}\cap P|}.$$
\end{proof}

We conclude this subsection with some results concerning Sylow subgroups.

\begin{lemma}\label{l: sylow odd order}
Let $G$ be a group and let $H\lhd G$ with $|H|$ odd. Suppose that $g$ is an involution in $P$, a $2$-subgroup of $G$. Define $\phi:G\to G/H$ to be the natural projection map. Then
$$|g^G\cap P| = |(gH)^{G/H}\cap \phi(P)|$$
\end{lemma}
\begin{proof}
Clearly $\phi$ maps $g^G\cap P$ onto $(gH)^{G/H}\cap \phi(P)$. Suppose that $\phi(g_1)=\phi(g_2)$ for $g_1, g_2\in g^G\cap P$. Then $g_1H=g_2H$ and so $g_1^{-1}g_2\in H$, which has odd order. Since $g_1^{-1}g_2$ also lies in $P$, which has even order, we conclude that $g_1=g_2$. Thus $\phi|_{g^G\cap P}$ is 1-1 and the result follows. 

\end{proof}

The next result is \cite[Theorem 1.5]{maroti}.

\begin{lemma}\label{l: maroti}
 Let $H$ be a nilpotent subgroup of the symmetric group $S_n$. Then $H$ contains at most $1.52^n$ conjugacy classes.
 \end{lemma}

\begin{lemma}\label{l: conj in sylow}
Let $q$ be an odd prime power. The number of conjugacy classes of involutions in a subgroup of $\GL_n(q)$ is strictly less than $3.04^n.$
\end{lemma}
\begin{proof}
Let $H$ be a subgroup of $\GL_n(q)$ and take $P$, a Sylow $2$-subgroup of $H$. Observe that $P$ contains a representative of every conjugacy class of involutions in $H$. Thus, to prove the result, we can assume that $H=P$, i.e. $H$ is a $2$-group. Let $P_n$ be a Sylow $2$-subgroup of $\GL_n(q)$ and assume that $H\leq P_n$.

Observe first that $P_{2k+1} = 2\times P_{2k}$ hence we obtain immediately that $C_{2k+1}\leq 2C_{2k}$ and it is sufficient to prove the result for $n$ even; assume this from here on.

Suppose next that $q\equiv 3\pmod 4$. Then a Sylow $2$-subgroup of $\GL_n(q)$ is a subgroup of $\GL_n(q^2)$ and $q^2\equiv 1\pmod 4$. Hence it is sufficient to assume that $q\equiv 1\pmod 4$. In this case $H$ is a subgroup of $(q-1)^n: S_n$; we take $H$ to be a subgroup of $\langle D, M\rangle$ where $D$ is the group of invertible diagonal matrices, and $M$ is the set of permutation matrices. 

Write $D_0$ for $D\cap H$ and observe first that there are at most $2^n-1$ involutions in $D_0$. Now consider a coset of $D_0$ in $H$, $gD_0$, where $g$ is an involution in $H\backslash D_0$. Choose a basis so that $g$ corresponds to permutation $(1 \, 2)(3 \, 4) \cdots (k-1, k)$ as follows:
$$g=\left(\begin{matrix}
            & g_1 && & \\
g_1^{-1} &   & & & \\
&&&g_3 & \\
&&g_3^{-1} && \\
&&&& \ddots 
          \end{matrix}\right).$$
Choose $d$ to be a diagonal matrix:
$$d=\left(\begin{matrix}
           d_1 & & \\
& d_2 & \\
& & \ddots
          \end{matrix}\right).$$
Now $D_0$ acts by conjugation on the coset $gD_0$ and observe that 
\begin{itemize}
 \item $dgd^{-1} = g$ if and only if $d_i=d_{i+1}$ for $i=1,3,\dots, k-1$;
\item $(dg)^2=1$ if and only if $d_id_{i+1}=1$ for $i=1,3,\dots, k-1$ and $d_{k+1},\dots, d_n=\pm 1$.
\end{itemize}
Let $D_1=\{d\in D_0: (dg)^2=1\}$, a subgroup of $D_0$. Observe that $|C_{D_1}(g)|\leq 2^{n-\frac{k}2}$.  What is more, since $D_0$ is abelian, $C_{D_1}(g) = C_{D_1}(gd)$ for any $d\in D_0$, thus every orbit of $D_1$ on the coset $gD_0$ must have size at least $\frac{|D_1|}{2^{n-\frac{k}2}}$. In other words the number of $D_1$-conjugacy classes of involutions represented in the coset $gD_0$ is at most $2^{n-\frac{k}2}<2^n$.

It is therefore clear that the number of conjugacy classes of involutions in $H$ is strictly less than $2^n\times (C'_n+1)$ where $C'_n$ is the maximum number of conjugacy classes of involutions in a subgroup of $S_n$. Lemma~\ref{l: maroti} implies that $C'_n+1\leq 1.52^n$ and the result follows.
\end{proof}

Lemmas \ref{L: sylowtwos} and \ref{l: conj in sylow} imply that if $g\in N\lhd H<\GL_n(q)$ with $g^2=1$, then
$$\frac{|g^H\cap P|}{|g^N\cap P|}<3.04^n$$
for $P$ a Sylow $2$-subgroup of $N$.

\subsection{The projective plane situation}\label{SS: pp}

This subsection is the last, until Section \ref{S: quaternionsylowtwos}, in which we will directly consider projective planes. Hence all the notation in this subsection is self-contained and we will develop this notation as we go along.

We begin by stating a hypothesis which will hold throughout this subsection. The conditions included represent, by \cite{wagner} and \cite[4.1.7]{dembov}, the conditions under which a group may act transitively on the points of a non-Desarguesian projective plane.

\begin{hypothesis}\label{h:basic}
Suppose that a group $G$ acts transitively upon the points of a non-Desarguesian projective plane $\spaceP$ of order $x>4$. If $G$ contains any involutions then each fixes a Baer subplane; in particular they fix $u^2+u+1$ points where $x=u^2, u>2$. Furthermore in this case any non-trivial element of $G$ fixing at least $u^2$ points fixes either $u^2+u+1, u^2+1$ or $u^2+2$ points of $\spaceP$.
\end{hypothesis}

We collect some significant facts that follow from Hypothesis \ref{h:basic}:

\begin{lemma}\label{L:firstly}
The Fitting group and the generalized Fitting group of $G$ coincide, i.e. $F^*(G)=F(G).$ What is more, $F(G)$ acts semi-regularly on the points of $\spaceP$. Further, if $p$ is a prime dividing $u^2+u+1,$ then $p\equiv 1\pmod3$ or $p=3$ and $9$ does not divide $u^2+u+1$.
\end{lemma}
\begin{proof}
The results about the Fitting group of $G$ can be found in \cite[Theorem A]{gill2} and \cite[Theorem A]{gill4}. The result about prime divisors can be found in \cite[p.33]{kantor}.
\end{proof}

Note that $u^2-u+1=(u-1)^2+(u-1)+1$ hence Lemma \ref{L:firstly} implies that if $p$ is a prime dividing $u^2-u+1,$ then $p\equiv 1\pmod3$ or $p=3$ and $9$ does not divide $u^2-u+1$.

Write $\alpha$ for a point of $\spaceP$. For integers $k$ and $w$, write  $k_w$ (resp. $k_{w'}$) for the largest divisor of $k$ which is a power of $w$ (resp. coprime to $w$). We write $\Fix(g)$ for the set of fixed points of $g$; similarly $\Fix(H)$ is the set of fixed points of a subgroup $H<G$. We begin with the following observation:

\begin{lemma}\label{L: threepossibilities}
Suppose that $G$ contains an involution. Then one of the following holds:
\begin{enumerate}
\item the Sylow $2$-subgroups of $G$ are cyclic or generalized quaternion;

\item All primes which divide $|F(G)|$ must also divide $u^2+u+1$.
\end{enumerate}
\end{lemma}
\begin{proof}
Suppose that the Sylow $2$-subgroups of $G$ are not cyclic or generalized quaternion. This is equivalent to supposing that the Sylow $2$-subgroups of $G$ contain a Klein $4$-group. Let $N$ be a Sylow $r$-subgroup of $F(G)$, for some prime $r$, and observe that any subgroup of $G$ acts by conjugation on $N$. Then \cite[(40.6)]{aschbacher3} implies that a Sylow $2$-subgroup of $G$ does not act semi-regularly on $N$. Hence $G$ contains an involution $g$ for which $C_N(g)$ is non-trivial. Now $C_N(g)$ acts on $\Fix(g)$, a set of size $u^2+u+1$. Because $F(G)$ acts semi-regularly on the points of $\spaceP$ we conclude that $r$ divides $u^2+u+1$.
\end{proof}

We will be interested in the second of these possibilities. So for the rest of this subsection we add the following to our hypothesis:
\begin{hypothesis}\label{h:main}
Suppose that $G$ contains an involution and that all primes dividing $|F(G)|$ also divide $u^2+u+1$. 
\end{hypothesis}

Clearly if we can show that Hypotheses 1 and 2 lead to a contradiction then Lemma \ref{L: threepossibilities} will imply Theorem A. Over the rest of this subsection we will work towards showing that, provided Lemma A is true, such a contradiction does indeed follow from these hypotheses. (We will state Lemma A shortly.)

Write $u^2+u+1=p_1^{a_1}\cdots p_r^{a_r}$ (a product of prime powers) and observe that, by Lemma \ref{L:firstly}, $p_i\equiv 1\pmod3$ or else $p_i^{a_i}=3$. Let $F(G)=N_1\times N_2\times\dots\times N_r$ where $N_i\in Syl_{p_i} F(G)$ and set $Z=Z(F(G))$. Since $F^*(G)=F(G)$, we conclude that $G/Z$ is a subgroup of 
$$Aut N_1\times\dots\times Aut N_r$$
(see, for instance, \cite[(31.13)]{aschbacher3}). Write $V_i=N_i/\Phi(N_i)$. Here $\Phi(N_i)$ is the Frattini subgroup of $N_i$, hence $V_i$ is a vector space over the field of size $p_i$. Observe that $|V_i|\leq p_i^{a_i}$.

A classical result of Burnside (see, for instance, \cite[Theorem 1.4]{gorenstein3}) tells us that $Aut N_i$ acts on $V_i$ with kernel, $K_i$, a $p_i$-group. Thus $G/Z$ is a subgroup of
$$K_1.\GL(V_1)\times\dots\times K_r.\GL(V_r).$$
Let $K^\dagger =(K_1\times\dots\times K_r) \cap G/Z$ and take $K$ to be the pre-image of $K^\dagger$ in $G$. Thus $G/K$ is a subgroup of
$$\GL(V_1)\times\dots\times \GL(V_r).$$

\begin{lemma}\cite[Lemma 13]{gill2}\label{L: counting}
Let $x=u^2$ and let $g$ be an involution in $G$. Then
$$\frac{|g^G|}{|g^G\cap\Galph|}=u^2-u+1.$$
\end{lemma}

\begin{lemma}\label{L: divide}
Let $g$ be an involution in $G$. Then $u^2-u+1$ divides
$$|G/K:C_{G/K}(gK)|.$$
\end{lemma}
\begin{proof}
By Lemma \ref{L: counting}, $u^2-u+1$ divides $|G:C_G(g)|$. Then, by Lemma \ref{L: oddnormal},
$$|G:C_G(g)|=|K:C_K(g)|\times|G/K:C_{G/K}(gK)|.$$
Now all primes dividing $|K|$ also divide $u^2+u+1$. But $u^2+u+1$ is coprime to $u^2-u+1$ so we have our result.
\end{proof}

We wish to apply Lemma \ref{L: invcentralizer} to the group $G/K$ which is a subgroup of $\GL(V_1)\times\dots\times \GL(V_r).$ In applying Lemma \ref{L: invcentralizer} we take $H=G/K$. For $i=1,\dots, r$,
\begin{itemize}
\item $H_i=\GL(V_i)$,
 \item $L_i$ is the projection of $H$ to $\GL(V_i)\times \GL(V_{i+1})\times\dots\times \GL(V_r)$.
 \end{itemize}
 For $i=1,\dots, r-1$,
 \begin{itemize}
 \item $\psi:L_i\to L_{i+1}$ is the natural projection, and
 \item $T_i$ equals $\ker(\psi_i)$.
\end{itemize}

We can order the $V_i$ so that $|T_i|$ is odd for $i<k$ and $|T_i|$ is even for $i\geq k$ where $k$ is some integer. We adopt such an ordering; furthermore we choose such an ordering for which $k$ is as large as possible. If $k<r$ then we also wish to guarantee that
$$p_k^{a_k-1}+\dots+p_k+1< p_{k+1}^{a_{k+1}-1}+\dots+p_{k+1}+1.$$
If this is not the case then we simply swap $V_k$ and $V_{k+1}$ in our ordering. Thus we assume that the inequality holds.

Now we are dealing with an involution $g$ in $G$ but we will apply Lemma \ref{L: invcentralizer} not to $g$ but $gK$ in $G/K$. Then $g_i$ will be the projection of $gK$ onto $\GL(V_i)\times\dots\times \GL(V_r)$. Recall that $u^2+u+1=p_1^{a_1}\cdots p_r^{a_r}$ where the $p_i$ are prime numbers; furthermore $|V_i|=p_i^{b_i}$ for some $b_i\leq a_i$. Finally, in order to state the next lemma concisely, we define $T_r$ to equal $L_r$.

\begin{lemma}\label{L: centralizerbound}
There exists $j\leq k$ such that
 $$(|T_j:C_{T_j}(g_j)|, u^2-u+1)>p_j^{a_j-1}+\dots+p_j+1.$$
\end{lemma}
\begin{proof}
We suppose that the proposition does not hold and seek a contradiction. In other words we suppose that, for all $i\leq k$,
$$(|T_i:C_{T_i}(g_i)|, u^2-u+1)\leq p_i^{a_i-1}+\dots+p_i+1$$

If $k=r$, then $T_r=L_r$ and
Lemmas \ref{L: invcentralizer} and \ref{L: divide} imply that $u^2-u+1$ divides
 \[\prod_{i=1}^{r}|T_i:C_{T_i}(g_i)|.\]
By our supposition this implies that
$$u^2-u+1\leq\prod_{i=1}^r(p_i^{a_i-1}+\dots+p_i+1)<\frac{u^2+u+1}{2^r}.$$
This is clearly a contradiction.

Now suppose that $k<r$. Since $P$ is a Sylow $2$-subgroup of $T_k$ which is isomorphic to a subgroup of $\GL(V_k),$ Lemma \ref{l: conj in sylow} implies that
$$\left(u^2-u+1, \frac{|(g_k)^{L_k}\cap P|}{|(g_k)^{T_k}\cap P|}\right)<3.04^{a_k}.$$
Now $3.04^k< p_k^{a_k-1}+\dots+p_k+1$ except when
\[
 (P_k,a_k) \in \{(3,1), (7,1), (7,2)\}
\]
Thus, ignoring these exceptions, our supposition implies that
$$u^2-u+1\leq\left(\prod_{i=1}^{k}(p_i^{a_i-1}+\dots+p_i+1)\right)(p_k^{a_k-1}+\dots+p_k+1).$$
Now recall that we have chosen an ordering such that $p_k^{a_k-1}+\dots+p_k+1<p_{k+1}^{a_{k+1}-1}+\dots+p_{k+1}+1$. This implies that
$$u^2-u+1\leq\prod_{i=1}^{k+1}(p_i^{a_i-1}+\dots+p_i+1)<\frac{u^2+u+1}{2^{k+1}}$$
and, once again, we have a contradiction. 

We must deal with the exceptions. When $a_k=1$, $\GL(V_k)$ is cyclic and so $|(g_k)^{L_k}\cap P|=1$ and the result follows immediately. When $(p_k, a_k)=(7,2)$, one simply uses the fact that 
\[
(u^2-u+1,|\GL(V_k)|)\leq 3 < p_k+1
\]
and we are done.
\end{proof}




\begin{proposition}\label{prop:implication}
Hypothesis 1 and Lemma A imply that Hypothesis 2 is false. 
\end{proposition}
\begin{proof}
Suppose Hypotheses 1 and 2 are true. We show that Lemma A leads to a contradiction. 
Define $K, V_i, L_i, T_i$ as in the lead up to Lemma~\ref{L: centralizerbound}.
In particular $G/K$ is a subgroup of $\GL(V_1)\times \dots\times \GL(V_r)$, where we order the $V_i$ as in Lemma \ref{L: centralizerbound} and recall that $|V_i|=p_i^{b_i}$ where $b_i\leq a_i$. 

We first apply Lemma A to $T_k$ which can be thought of as a subgroup of $\GL(V_k)$. Thus let $g_k$ be an involution in $T_k$ such that $|T_k:C_{T_k}(g_k)|_{p_k'}$ is minimised. If $p_k=3$, then $a_k=1$ and $|T_k:C_{T_k}(g_k)|_{p_k'}=1$. If $p_k>3$, then Lemma A implies that
$$|T_k:C_{T_k}(g_k)|_{p_k'}\leq p_k^{b_k-1}+\dots+p_k+1\leq p_k^{a_k-1}+\dots+p_k+1.$$
Now, by Lemma \ref{L:firstly} (and the comment immediately after), this implies that
$$(|T_k:C_{T_k}(g_k)|, u^2-u+1)\leq |T_k:C_{T_k}(g_k)|_{p_k'}\leq p_k^{a_k-1}+\dots+p_k+1.$$

Let $gK$ be an involutory pre-image of $g_k$ in $G/K$ (such a pre-image must exist since $|T_i|$ is odd for $i<k$); as before take $g_i$ to be the projection of $gK$ onto $\GL(V_i)\times \dots\times \GL(V_r)$. Observe that, in all cases $\langle T_j, g_j\rangle$ is isomorphic to a subgroup of $\GL(V_j)$.

Now, when $j<k$, $T_j$ has odd order. Then $H_j=\langle T_j, g_j\rangle$ has a unique $H_j$-conjugacy class of involutions. Thus Lemma A implies that, for all $j<k$ where $p_j\neq 3$,
$$(|T_j:C_{T_j}(g_j)|, u^2-u+1)\leq p_j^{a_j-1}+\dots+p_j+1.$$
Indeed the same bound holds when $p_j=3$ since in this case $a_j=1$. This yields a contradiction to Lemma \ref{L: centralizerbound}. 

Thus we have demonstrated that, provided Lemma A is true, our Hypotheses 1 and 2 lead to a contradiction. 
\end{proof}

\begin{corollary}
Lemma A implies Theorem A.
\end{corollary}
\begin{proof}
This is immediate from Lemma \ref{L: threepossibilities} and Proposition \ref{prop:implication}.
\end{proof}

\section{Proving Lemma A}\label{S: progress}

Throughout this section we occupy ourselves with a proof of Lemma A. Lemma A is a purely group theoretic result; we will not refer to projective planes in this section. We use standard group theory notation, as described at the start of Subsection \ref{S: background}. Note that Lemma A was stated on p.\pageref{la}.

Throughout this section $q=p^a$ with $p\geq 7$ and $H$ is a subgroup of $\GL_n(q)$. Suppose that $H$ lies in a maximal subgroup $M$ of $\GL_n(q)$. In order to prove Lemma A we will go through the possibilities for $M$ and $H$ and demonstrate that, in all cases, Lemma A holds. 

\subsection{The maximal subgroups of $\GL_n(q)$}

Before we embark on a proof of Lemma A we need to outline a classification of the maximal subgroups of $G=\GL_n(q)$. Such a classification was outlined in \cite{aschbacher2}; we follow the treatment of this result given in \cite{kl}. We take $V$ to be an $n$-dimensional vector space over $\Fq$. Let $\kappa$ be a sesquilinear (resp. quadratic) form defined over $V$. We will assume that $\kappa$ is either non-degenerate, or else is equal to the zero form. Define
\begin{equation*}
\begin{aligned}
\Gamma &= \{g\in \Gamma L(V) \, | \, \kappa (vg) = \lambda(g)\kappa(v)^{\alpha(g)} \text{ for all } v\in V\}; \\
\Delta &= \{g\in \GL(V) \, | \, \kappa (vg) = \lambda(g)\kappa(v) \text{ for all } v\in V\}; \\
S &= \{g\in SL(V) \, | \, \kappa (vg) = \kappa(v) \text{ for all } v\in V\};
\end{aligned}
\end{equation*}
where $\lambda(g)\in \Fq^*$ and $\alpha(g)\in Aut(\Fq)$. In other words $\Gamma$ (resp. $\Delta$, $S$) is the set of {\it semisimilarities} (resp. {\it similarities}, {\it special isometries}) of $\kappa$. We record the groups we get for various choices of $\kappa$ \cite[Tables 2.1.B and 2.1.D]{kl}:
\begin{center}
\begin{tabular}{|c|c|c|c|}
 \hline
$\kappa$ & $\Gamma$ & $\Delta$ & $S$ \\
\hline
Zero & $\Gamma L_n(q)$ & $\GL_n(q)$ & $\SL_n(q)$ \\
Hermitian & $\Gamma U_n(q)$ & $GU_n(q).(q-1)$ & $SU_n(q)$ \\
Alternating & $\Gamma Sp_n(q)$ & $GSp_n(q)$ & $Sp_n(q)$ \\
Orthogonal & $\Gamma O^\epsilon_n(q)$ & $GO_n^\epsilon(q)$ & $SO_n^\epsilon(q)$ \\
\hline
\end{tabular}
\end{center}

Here $\epsilon$ gives the type of the orthogonal space; it takes the value $\pm$ when $n$ is even, and is blank when $n$ is odd. The group $\Gamma$ contains the scalars as a normal subgroup, which we denote $Z$. For $H\leq \Gamma$ we write $\overline{H}$ for reduction modulo scalars: $\overline{H}=H/(H\cap Z)$.

\begin{lemma}\label{l: when omega is simple}\cite[Prop. 2.9.2]{kl}
$\overline{S}$ is simple except in the following cases:
\begin{enumerate}
 \item $n=1$ or $(n,q)\in\{(2,2), (2,3)\}$, or $S=SU(3,2)$, $Sp(4,2)$, $SO_2^{\pm}(q)$, $SO_3(3)$, or $SO_4^+(q)$.
\item $\kappa$ is quadratic, and (1) does not hold. In this case $\overline{S}$ contains a simple subgroup of index at most $2$.
\end{enumerate}
\end{lemma}

We will call situation (1) the {\it small rank setting}. Thus the small rank setting is described by particular values of $(n,q)$, and particular types of $\kappa$ (strictly speaking, since $p\geq 7$, most of the listed cases do not occur). When we are not in the small rank setting, we write $\Omega$ (or $\Omega_\kappa$) for the subgroup of $S$ of index at most $2$ such that $\overline{\Omega}$ is simple. Now \cite[Prop. 2.9.2]{kl} implies that $\Omega$ is quasisimple.

For particular choices of $\kappa, n$ and $q$ we are interested in the maximal subgroups of $X$, a group satisfying $\Omega\leq X \leq \Gamma$; most of the time we will apply the following results with $\kappa$ the zero form and $X=\Delta_\kappa\cong \GL_n(q)$.

\begin{lemma}\label{l: easy}
Suppose we are not in the small rank setting for $(\kappa,n,q)$. Let $X$ be a group satisfying $\Omega\leq X\leq \Gamma$.  If $M$ is a maximal subgroup of $X$ then either $M$ contains $\Omega$ or $\overline{M}$ is a maximal subgroup of $\overline{X}$. Conversely, if $M_1$ is a maximal subgroup of $\overline{X}$ then $M_1 = \overline{M}$ for some maximal subgroup $M$ in $X$.
 \end{lemma}
\begin{proof}
The converse statement is immediate: take $M$ to be the full pre-image of $M_1$ in $X$. Then $M$ must be a proper subgroup of $X$ and is clearly maximal.

Now suppose that $M$ is maximal in $X$. Then $\overline{M}$ must be a maximal subgroup of $\overline{X}$ or else $\overline{M}\cong \overline{X}$. Suppose the latter case holds. Then $M$ contains a perfect subgroup $M_0$, which has index at most $|X:\Omega|$ in $X$. Since $M_0$ is perfect and $\Omega$ is quasisimple we conclude that $M_0$ contains $\Omega$ or else $M_0\cap \Omega\leq Z(\Omega)$. But $|M_0|\geq \overline{\Omega}>|\Gamma:\Omega|\cdot |Z(\Omega)|$ (see \cite[Tables 2.1.C and 2.1.D]{kl}) and so we conclude that $M_0$ contains $\Omega$.
\end{proof}

The maximal subgroups of $\overline{X}$ are described in \cite[Theorem 1.2.1]{kl}. In order to describe them we start with a family $\curly(\Gamma)$ of subgroups of $\Gamma$ (we will describe these in due course). We define
\begin{equation*}
 \begin{aligned}
  \curly(X) &= \{X\cap H \, | \, H\in \curly(\Gamma)\}; \\
\curly(\overline{X}) &= \{\overline{H} \, | \, H\in\curly(X)\}.
 \end{aligned}
\end{equation*}

Now we have the following result, originally due to Aschbacher \cite{aschbacher2}:
\begin{theorem}\cite[Theorem 1.2.1]{kl}\label{t: pglmax}
Suppose we are not in the small rank setting for $(\kappa,n,q)$. Let $X$ be a group satisfying $\Omega\leq X\leq \Gamma$. If $H$ is a subgroup of $\overline{X}$ then either $H$ is contained in a member of $\curly(\overline{X})$, or else $H$ is almost simple with socle $S$ such that if $L$ is the full covering group of $S$, and if $\rho: L\to \GL(V)$ is a representation of $L$ such that $\overline{\rho(L)}=S$, then $\rho$ is absolutely irreducible. 
\end{theorem}

\begin{corollary}\label{c: general delta}
Suppose we are not in the small rank setting for $(\kappa,n,q)$. Let $X$ be a group satisfying $\Omega\leq X\leq \Gamma$. If $H$ is a subgroup of $X$ then one of the following holds:
\begin{enumerate}
 \item $H$ contains $\Omega$;
\item $H$ is contained in $M_1Z$ where $M_1\in\curly(X)$;
\item $H_1\leq H \leq H_1Z$ where $H_1$ is an almost quasisimple group such that $F^*(H_1)$ is absolutely irreducible on $V$. 
\end{enumerate}
\end{corollary}
\begin{proof}
Suppose that $H$ does not contain $\Omega$. Then Lemma \ref{l: easy} implies that $\overline{H}$ is a proper subgroup of $\overline{X}$ and we apply Theorem \ref{t: pglmax}. If $\overline{H}$ does not lie in a member of $\curly(\overline{X})$ then $\overline{H}$ is almost simple, and the third possibility holds. Otherwise $\overline{H}$ lies inside a member $\overline{M}$ of $\curly(\overline{X})$. Choose $\overline{M}$ to be maximal; Lemma \ref{l: easy} implies that $H$ lies inside a maximal subgroup of $X$ equal to the full pre-image of $\overline{M}$. Now $\overline{M}=\overline{M_1}$ for some $M_1\in\curly(X)$. Thus the full pre-image of $\overline{M}$ is equal to $M_1Z$ as required.
\end{proof}

Our job now is to describe $\curly(\Gamma)$, and hence $\curly(X)$. Each of these splits into eight subfamilies: $\curly_i(\Gamma)$ (resp. $\curly_i(X)$), for $i=1,\dots, 8$. We now sketch a description of these; full descriptions (with the same terminology) can be found in \cite{kl}. We are particularly interested in the case where $X=\Delta$ hence we introduce the following notation: take $H_\Gamma\in \curly_i(\Gamma)$ and write $H=H_\Delta = H_\Gamma\cap \Delta$, for the corresponding member of $\curly_i(\Delta)$.

\begin{itemize}
 \item[$\curly_1$:] $H_\Gamma = N_\Gamma(W)$ where $W$ is a subspace of $V$ of dimension $m\leq\frac{n}{2}$ that is either non-degenerate or totally singular.

\item[$\curly_2$:] Write $n = mt$, for $t\geq 2$. $H_\Gamma$ is the stabilizer in $\Gamma$ of an {\it $m$-space decomposition}. If $\kappa=0$, then this is a decomposition $V=V_1\oplus V_2\oplus \cdots \oplus V_t$ where $V_i$ is an $m$-dimensional subspace of $V$; if $\kappa$ is non-degenerate, then one must impose extra conditions on the subspaces $V_i$. In any case $H$ is the stabilizer in $\Delta$ of this $m$-space decomposition.

\item[$\curly_3$: ] Write $n=mr$ for $r$ prime. $H_\Gamma = \Gamma_{\#,\mu}$, a large subgroup of the group $\Gamma_\mu$ corresponding to a form $\mu$ on $V$ where $V$ is viewed as a vector space over $\Fqr$ (the precise description of this subgroup is given on \cite[p. 111]{kl}). The form $\mu$ depends on $\kappa$, as described in \cite[Table 4.3.A]{kl}; for instance if $\kappa=0$ then $\mu=0$. Then $H = \Gamma_{\#,\mu}\cap \Delta$.

\item[$\curly_4$: ] Write $n=n_1n_2$ with $n_1\geq 2$, and consider $V_1\otimes V_2$, where $V_1$ (resp. $V_2$) is a vector space of dimension $n_1$ (resp. $n_2$). This induces an imbedding 
$$\GL(V_1)\otimes \GL(V_2)\cong \GL_{n_1}(q)\circ \GL_{n_2}(q)\leq \GL(V)\cong \GL_n(q).$$
We can define forms on $V_1$ and $V_2$, $\kappa_1$ and $\kappa_2$, that combine to give a form on $V_1\times V_2$; if chosen carefully this combination of forms may be identified with the form $\kappa$. Thus we induce an imbedding $\Delta_{\kappa_1}\otimes \Delta_{\kappa_2}\leq \Delta_{\kappa}$. We define $H_\Gamma = (\Delta_{\kappa_1}\otimes \Delta_{\kappa_2})\langle\phi\rangle$, where $\phi$ is a field automorphism; then $H=\Delta_{\kappa_1}\otimes \Delta_{\kappa_2}$ \cite[(4.4.13)]{kl}.

\item[$\curly_5$: ] Let $V_0$ be an $n$-dimensional vector space over $\mathbb{F}_0$, a subfield of $\Fq$. Then $V$ can be identified with $V_0\otimes \Fq$, and a form, $\kappa_0$, on $V_0$ can be extended to a form $\kappa$ on $V$. We then define $H_\Gamma=N_\Gamma(V_0)Z$; in particular $H=\Delta_0Z$ \cite[(4.5.5)]{kl}.

\item[$\curly_6$: ] Write $n=r^m$, for $r$ prime and not equal to $p$. Then $H_\Gamma=N_\Gamma(R)$ where $R$ is a particular symplectic-type $r$-subgroup of $\Delta$.

\item[$\curly_7$: ] Write $n=m^t$. The groups here are similar to those in $\curly_4$; this time we identify $V$ with a tensor product $V_1\otimes V_2\otimes\cdots\otimes V_t$ where $V_i,i=1,\dots, t$ is a vector space of dimension $m$. Then $H_\Gamma = (\Delta_{\kappa_1}\circ \cdots \circ \Delta_{\kappa_t})(\langle \phi\rangle \times J)$, where $\phi$ is a field automorphism, and $J\cong S_t$ permutes the $V_i$. We conclude that $H=(\Delta_{\kappa_1}\circ \cdots \circ \Delta_{\kappa_t})J$ \cite[p.157]{kl}.

\item[$\curly_8$: ] Here $H_\Gamma = \Gamma_\mu$ where $\mu$ is a non-degenerate form defined on $V$. Then $H=\Delta_\mu$. The form $\mu$ depends on $\kappa$, as described in \cite[Table 4.8.A]{kl}.
\end{itemize}

We record two significant corollaries.

\begin{corollary}\label{c: breakdown}
Fix $n>2$ and $q=p^a$ with $p\geq 7$, and set $\kappa$ to be the zero form. Let $X$ be a group satisfying $\Delta\leq X\leq \Gamma$. If $H$ is a subgroup of $X$ then one of the following holds:
\begin{enumerate}
\item $H$ contains $\Omega$.
\item $H$ is contained in $M_1Z$ where $M_1\in\curly_i(X)$, for $i=1,\dots,7$.
\item $H$ is contained in $M_1Z$ where $M_1\in\curly_6(X_\xi)$ for some non-degenerate form $\xi$ acting on $V$.
\item $H_1\leq H\leq  H_1Z$ where $H_1$ is almost quasisimple and $F^*(H_1)$ is absolutely irreducible on $V$.
\item $H$ lies in $X_\xi$, for some non-degenerate form $\xi$ acting on $V$ and $(\xi, n,q)$ lies in the small rank setting.
\end{enumerate}
\end{corollary}

Note that, by $X_\xi$, we mean a group satisfying $\Delta_\xi\leq X_\xi \leq \Gamma_\xi$, for some non-degenerate form $\xi$ acting on $V$. Note too that the requirement that $\kappa$ be the zero form is somewhat superfluous; we retain this condition as it makes the corollary easier to state.

\begin{proof}
The result differs from Corollary \ref{c: general delta} by specifying what happens when $H$ is contained in $M_1Z$ for $M_1\in\curly_8(X)$. So let us examine this case; here $M_1=X_\xi$ for some non-degenerate form $\xi$ acting on $V$; in fact $\Delta_\xi$ contains $Z$ so we conclude that $H\leq X_\xi$. We can now apply Corollary \ref{c: general delta} again. If we lie in the small rank setting for $(\xi, n,q)$ then (5) holds. Assume that we do not lie in the small rank setting. If $H$ contains $\Omega_\xi$ then \cite[Proposition 2.10.6]{kl} implies that (4) holds. If $H$ contains an almost-simple group $H_1$ with $F^*(H_1)=S_1$ absolutely irreducible on $V$ then, again, (4) holds. 

Finally we have the possibility that $H$ lies in $M\in\curly(X_\xi)$. If $M\in\curly_6(X_\xi)$, then (3) holds. If $M\in\curly_i(X_\xi)$, for $i=1,\dots,5, 7$, then the descriptions that we have given above show that $H$ must lie in $\curly_i(X_\kappa)$ and (2) holds. (Observe that, if $M\in\curly_i(X_\xi)$, for some $i=1,\dots, 5,7$, then $M$ stabilizes some geometric configuration in $V$, as well as preserving $\xi$. Consider the group $M_1$, the full stabilizer of this geometric configuration in $X_\kappa$. Then $M_1\in\curly_i(X_\kappa)$ as required.) 

Finally, if $\xi$ is alternating and $q$ is even then it is possible that $M=X_\psi\in\curly_8(X_\xi)$, where $\psi$ is a non-degenerate quadratic form. In this case we repeat the above argument, and the result follows.
\end{proof}

\begin{corollary}\label{c: breakdown2}
 Suppose that $M$ is a maximal subgroup of $\Delta=\GL_n(q)$ lying in one of the families $\curly_i(\Delta)$, $i=1,\dots, 8$. Then $M$ is isomorphic to one of the following groups ($t,m,r$ are all integers; $r$ is prime).
\begin{itemize}
\item[$\curly_1$: ] $[q^{m(m-n)}]. (\GL_m(q)\times \GL_{n-m}(q))$, for $1\leq m<n$.
\item[$\curly_2$: ] $\GL_m(q) \wr S_t$, for $n=mt$ with $t\geq 2$.
\item[$\curly_3$: ] $\GL_m(q^r).r$, for $n=mr$, and the extension is given by a field automorphism.
\item[$\curly_4$: ] $\GL_{n_1}(q)\circ \GL_{n_2}(q)$, for $n=n_1n_2$ with $2\leq n_1 < \sqrt{n}$.
\item[$\curly_5$: ] $\GL_m(q_0)Z$ for $q=q_0^r$ and $Z=Z(\GL_n(q))$.
\item[$\curly_6$: ] The normalizer of an absolutely irreducible symplectic-type $r$-group, for $n=r^m$ with $r\neq p$.
\item[$\curly_7$: ] $\underbrace{(\GL_m(q)\circ \cdots \circ \GL_m(q))}_t.S_t$ for $n=m^t, m\geq 3$.
\item[$\curly_8$: ] $\Delta_\xi$ for $\xi$ a non-degenerate form on the associated vector space $V$.
\end{itemize}
\end{corollary}
\begin{proof}
This is immediate from the descriptions of $\curly_i(\Delta)$ given above. (For the most part descriptions are given in \cite{kl}.)
\end{proof}

\subsection{Preliminaries for a proof of Lemma A}

We now have enough information about subgroups of $\GL_n(q)$ to begin a proof of Lemma A. Let us remind ourselves of the statement:

\begin{lemmaa}
Suppose that $n$ is a positive integer and that $q=p^a$ for some integer $a$ and prime $p\geq 7$. Let $H$ be an even order subgroup of $\GL_n(q)$. Then there exists an involution $g\in H$ such that
$$|H:C_H(g)|_{p'}\leq q^{n-1}+\dots+q+1.$$
\end{lemmaa}

Let us begin with a proof of Lemma A for $n=2$ (when $n=1$, the statement is a trivial consequence of the fact that $\GL_1(q)$ is abelian).

\begin{lemma}\label{l: psl2}
A maximal subgroup of $\PSL_2(q)$, for $q$ odd, is isomorphic to one of the following groups:
\begin{enumerate}
 \item $D_{q\pm 1}$ (normalizers of tori);
\item $[q]:(\frac{q-1}2)$ (a Borel); 
\item $S_4, A_4$, or $A_5$;
\item $\PSL_2(q_0)$, for $q=q_0^a$ where $a$ is an odd prime;
\item $\PGL_2(q_0)$, for $q=q_0^2$.
\end{enumerate}
\end{lemma}
\begin{proof}
This is well-known; see, for example, \cite{dickson}.
\end{proof}

\begin{lemma}\label{l: gl2}
Let $H<\GL_2(q)$ and suppose that $H$ has even order. Suppose that $q=p^a$ with $p\geq 7$. Then there exists an involution $g\in H$ such that
$$|H:C_H(g)|_{p'}\leq q+1.$$
\end{lemma}
\begin{proof}
Observe that $|H:C_H(g)|_{p'}$ divides $|H/(H\cap Z(\GL_2(q))\cap \PSL_2(q)|_{p'}$ and that $H/(H\cap Z(\GL_2(q))\cap \PSL_2(q)$ is a subgroup of $\PSL_2(q)$. If $H/(H\cap Z(\GL_2(q))\cap \PSL_2(q)=\PSL_2(q)$, then $H\geq \SL_2(q)$, $H$ contains a central involution, and the result follows. If $H/(H\cap Z(\GL_2(q))\cap \PSL_2(q)< \PSL_2(q)$ it is sufficient to prove that $|H_1|_{p'}\leq q+1$ for all maximal subgroups $H_1$ of $\PSL_2(q)$. This follows from Lemma \ref{l: psl2}.
\end{proof}

The following result will be useful when we come to consider subgroups of $\GL_n(q)$ with a large normal odd-order subgroup.

\begin{lemma}\label{l: relevant groups}
Suppose that $H$ is a subgroup of $\GL_n(q)$ where $n>2$, $q=p^a$ and $p\geq 7$ is prime. If 
$$|H|_{(2p)'}\geq q^{n-1}+\dots+q+1,$$
then one of the following holds:
\begin{enumerate}
\item $H$ contains $\SL_n(q)$.
\item $H$ is contained in $M_1Z$ where $M_1\in\curly_i(\GL_n(q))$, for $i=1,\dots,5$.
\item $H_1\leq H\leq  H_1Z$ where $H_1$ is almost quasisimple and $F^*(H_1)$ is absolutely irreducible on $V$.
\item $n=4$ and $H\leq \GO_4^+(q)$.
\item $H\leq (\GL_m(q)\circ \GL_m(q)).2$ with $n=m^2$ for $m\geq 3$.
\item $(n,q)=(3,7)$ and $|H/(H\cap Z)|_{(2p)'}$ divides $27$.
\end{enumerate}
\end{lemma}
\begin{proof}
Since $H$ is a subgroup of $\GL_n(q)$ we apply Corollary \ref{c: breakdown}. If conclusion 5 of that corollary applies, then, since $p\geq 7$ and $n>2$, we must have conclusion 4 here.

Thus we are left with the question of what happens when $H\leq M\in \curly_6(\Delta_\zeta)$ for some form $\zeta$ that is non-degenerate or zero, or when $H\leq M\in \curly_7(\Delta_\kappa)$ when $\kappa$ is the zero form.

Consider the latter situation first. Then $M$ is equal to 
$$\underbrace{(\GL_m(q)\circ \cdots \circ \GL_m(q))}_t.S_t$$
for $n=m^t, m\geq 3$. Observe that $t!<7^{t^2}$ and so $|M|_{p'}< q^{m^2t} t!\leq q^{m^2t+t^2}<q^{m^t-1}$ unless $t=2$ or $(m,t)=(3,3)$. When $(m,t)=(3,3)$, observe that $|M|_{(2p)'}<3q^{18}$; when $t=2$ we obtain conclusion 5.

We are left with the problem of showing that if $M\in \curly_6(\Delta_\zeta)$ for some form $\zeta$ that is non-degenerate or zero, then either conclusion 6 holds or
$$|M|_{(2p)'}\leq q^{n-1}+\dots+q+1;$$
in fact we prefer to show that either conclusion 6 holds or $M_1 = (M/(M\cap Z))\cap \PSL_n(q)$ satisfies the inequality 
\[
|M_1|_{(2p)'}|(n,q-1)|_{(2p)'}<q^{n-2}.
\]
Observe that $M_1$ is a subgroup of $\PSL_n(q)$ and so its structure is given in \cite[\S 4.6]{kl}.

Suppose first that $n=2^m$. Then
\begin{equation*}
\begin{aligned}
|M_1|_{(2p)'}|(n,q-1)|_{(2p)'} &\leq \max\{45, |Sp_{2m}(2)|_{(2p)'}, |\Omega^+_{2m}(2)|_{(2p)'}, |\Omega^-_{2m}(2)|_{(2p)'}\} \\
&\leq  \max\{45, \prod_{i=1}^{m}(2^{2i-1} - 1)\} \\
&\leq \max\{45, 2^{m^2+m}\}.
\end{aligned}
\end{equation*}
Now $\max\{45, 2^{m^2+m}\}<q^{2^m-2}$ for $m\geq 3$. When $m=2$, we find that $|M_1|_{(2p)'} \leq 45$ and we are done.

Suppose that $n=r^m$ for $r$ odd. Then $|M_1|_{(2p)'}\leq \max\{9, |Sp_{2m}(r)|_{(2p)'}\}$ and 
\[|M_1|_{(2p)'}|(n,q-1)|_{(2p)'}<r^{2m^2+4m}.\]
This yields the required bound unless
\[
 (m,r) \in\{(2,3), (1,3), (1,5), (1,7).
\]
All of these can be ruled out individually, except $(m,r)=(1,3)$. In this case we obtain conclusion 6.
\end{proof}

In what follows we will prove Lemma A by referring to the five types of subgroup given in Corollary~\ref{c: breakdown}.

\subsection{Easy cases}

We begin by proving Lemma A for some easy situations; recall that $p\geq 7$. Apply Corollary~\ref{c: breakdown} with $X=\GL_n(q)$. Set $Z:=Z(\GL_n(q))$.

{\bf Case 1}. Here $H\geq \Omega = SL_n(q)$. If $n$ is even then $H$ contains a central involution and we are done. If $n$ is odd then $H$ contains an involution $g$ such that $|H:C_H(g)|_{p'} = q^{n-1}+\dots+q+1$.

{\bf Case 5}. In this case  $n=4$ and $H\leq \GO_4^+(q)$; now $\GO_4^+(q)$ has a normal subgroup $N$ of index $2$ isomorphic to 
\[\GL_2(q)\circ \GL_2(q)=\GL_2(q)\times \GL_2(q)/\{(g,g^{-1}) \mid g\in Z(\GL_2(q))\}. 
\]
Now we apply Lemma~\ref{l: gl2} to the group $N$ and the result follows.

{\bf Case 2 with $M_1\in \curlyc{7}(\GL_n(q))$}. In this case $M$ is equal to 
$$\underbrace{(\GL_m(q)\circ \cdots \circ \GL_m(q))}_t.S_t$$
for $n=m^t, m\geq 3$. Observe that $t!<7^{t^2}$ and so $|M|_{p'}< q^{m^2t} t!\leq q^{m^2t+t^2}<q^{m^t-1}$ unless $t=2$ or $(m,t)=(3,3)$. When $(m,t)=(3,3)$, observe that $|M|_{(p)'}<6q^{18}$ and we are done.

Finally suppose that $t=2$ and so $H\leq M=(\GL_m(q)\circ \GL_m(q)):2$ with $m\geq 3$. Define $N=\GL_m(q)\circ \GL_m(q)$, a normal subgroup in $M$ of index $2$. More precisely $N$ is equal to
\[
 \GL_m(q)\times \GL_m(q)/ \{(g,h) \in Z(\GL_m(q))\times Z(\GL_m(q)) \mid h=g^{-1}\}.
\]
Let $g$ be an involution in $M\backslash N$; then one can check that $g$ is the projective image of $((a_1, \pm a_1^{-1}), b)\in (\GL_m(q)\times \GL_m(q))\rtimes C_2$, with $b$ non-trivial. Easy matrix calculations confirm that
$|\GL_m(q)|$ divides $|M:C_M(g)|$, hence if $H$ contains an involution $g$ outside $N\cap H$ then $|H:C_H(g)|_{p'}<q^{\frac{m(m+1)}2}$ which satisfies our bound since $m\geq 3$.

If, on the other hand, $H$ contains no such involution then we take $g\in N$. The largest conjugacy class $C$ of involutions in $N$ has $|C|_{p'}$ divisible by
$$\left(\frac{\prod_{i=1}^m(q^i-1)}{\prod_{i=1}^{\lfloor\frac{m}2\rfloor}(q^i-1)\prod_{i=1}^{\lceil\frac{m}2\rceil}(q^i-1)}\right)^2.$$
Then $|H:C_H(g)|_{p'}\leq 2|N:C_N(g)|_{p'}<2q^{\frac12m^2}$ which satisfies our bound since $m\geq 3$.

{\bf Case 3 and Case 2 with $M_1\in \curlyc{6}(\GL_n(q))$.} In these cases $H\leq M\in\curly_6(\Delta_\eta)$ for $\eta$ non-degenerate or zero.

Suppose first that $n=r^m$ where $r$ is an odd prime. Then 
\[
\begin{aligned}
 |M_1/(M_1\cap Z) &\leq (n,q-1)\cdot \max\{72, r^{2m}|\Sp_{2m}(r)|\} \\ 
 &\leq (n,q-1)\cdot \max\{72, r^{3m+2m^2}\}.
\end{aligned}
 \]
One can check that $|M_1/(M_1\cap Z)|<q^{n-1}$ unless $n=3,5$ or $9$. If $n=5$ or $9$, then one must have $q=7$ and the result can be confirmed by inspection. If $n=3$, then $|M_1/(M_1\cap Z)|_{2'}$ divides $81$ and so any involution $g$ in the center of a Sylow $2$-subgroup of $H\leq M_1$ will satisfy the required bound unless $q=7$. If $n=3$ and $q=7$, then the result can be confirmed by inspection.

Suppose next that $n=2^m$ for some $m\geq 2$. Then one can check that $|M_1/(M_1\cap Z)|_{2'}<7^{n-2}$ and so any involution $g$ in the center of a Sylow $2$-subgroup of $H\leq M_1$ will satisfy the required bound.


\subsection{Case 4 (almost quasisimple)}

We need to prove Lemma A for $H$ where $H_1\leq H\leq  H_1Z$ and $H_1$ is almost quasisimple such that $F^*(H_1)$ is absolutely irreducible. Clearly we can assume that $H=H_1$.

In what follows $g\in H$ is an involution.
Suppose first that $Z=Z(F^*(H_1))$ has even order. Checking Schur multipliers \cite[Theorem 5.1.4]{kl}, we see that $Z$ contains at most $3$ involutions. Taking $g$ to be one of these we obtain that $|H: C_{H}(g)|$ divides $3$ and the result follows. Suppose from here on that $|Z|$ is odd.

Now Lemma~\ref{L: oddnormal} implies that
\[|H:C_H(g)| = |Z:C_Z(g)|\times |H/Z: C_{H/Z}(gZ)|.\]
Clearly if $g\in F^*(H_1)$, then
\[|H:C_H(g)| = |H/Z: C_{H/Z}(gZ)|.\]
Thus to prove Lemma A in this case it is sufficient to show that the simple group $S=F^*(H_1)/Z$ contains an involution $g$ such that 
\begin{equation}\label{ccc}
|\Aut(S): C_{\Aut(S)}(g)| _{p'} \leq q^{n-1}+\dots + q+1.
\end{equation}

Our first step in proving this bound is to establish, given a simple group $S$, what the minimum possible value for $n$ is. We observe first that an embedding of $H$ in $\GL(V)$ induces an imbedding of $H/Z_1$ in $\PGL(V)$ where $Z_1$ is some central subgroup of $F^*(H_1)$. Define
$$R_p(G)=\min\{n \, | \, G\leq \PGL_n(\mathbb{F}) \textrm{ for } \mathbb{F} \textrm{ a field of characteristic } p\}.$$

\begin{lemma}\cite[Corollary 5.3.3]{kl}
If $L$ is quasisimple, then $R_r(L)\geq R_r(L/Z(L))$ for all primes $r$.
\end{lemma}

The lemma implies that, given a simple group $S$, it is sufficient to prove that $S$ contains an involution $g$ satisfying \eqref{ccc}, where $n=R_p(S)$. The next lemma establishes this fact for a large class of simple groups.

\begin{lemma}
Let $p\geq 7$ be a prime, and let $S$ be a non-abelian simple group. Assume that $S$ is not isomorphic to $A_5$ and that $S$ is not a group of Lie type in characteristic $p$. Then $S$ contains an involution $g$ such that
$$|Aut(S):C_{Aut(S)}(g)|_{p'} \leq p^{n-1}+\dots + p+1,$$
where $n=R_p(S)$.
\end{lemma}
\begin{proof}
The lemma is proved using information in \cite{kl}. In particular we use Propositions 5.3.7 and 5.3.8 as well as Theorem 5.3.9, all of which list values of $R_p(S)$ for different non-abelian simple groups $S$.

Suppose that $S=A_m$ is the alternating group on $m$ letters with $m\geq 8$. Then $Aut(S)\cong S_m$ and $S$ contains an involution $g$ (a double transposition) such that 
$$|S_m: C_{S_m}(g)| = \frac{m(m-1)(m-2)(m-3)}{8}.$$
In particular $|S_m: C_{S_m}(g)|<7^{m-3}$ for $m\geq 8$. Now \cite[Proposition 5.3.7]{kl} implies that $n\geq m-2$ and so $|S_m: C_{S_m}(g)|<7^{m-3}\leq q^{n-1}$ as required. For $m=6,7$ one can check the result directly. 

If $S$ is a sporadic group then one can use \cite[Table 5.1.C and Proposition 5.3.8]{kl} along with \cite{atlas} to confirm the result in each case.
 
Next suppose that $S$ is a group of Lie type of characteristic coprime to $p$. Values for $R_p(S)$ are given in \cite[Theorem 5.3.9]{kl}, and a series of (easy and tedious) calculations confirms the inequality that we require. We present two sample calculations: If $H=E_6(r)$ then $R_p(S)=r^9(r^2-1)$. Now observe that $|S|<r^{78}$; it is then sufficient to observe that $r^{78}<7^{r^9(r^2-1)-1}$. If, on the other hand, $S=\Omega_{2m+1}(r)$ with $r$ odd and $m>3$ then $R_p(S)\geq r^{2(m-1)}-r^{m-1}$. Now $|S|<r^{2m^2+m}$ and it is sufficient to observe that $r^{2m^2+m}<7^{r^{2(m-1)}-r^{m-1}-1}$.\label{rechecked for $L_2(q)$}
\end{proof}

Let us deal with the case $A_5$.

\begin{lemma}
 Suppose that $H_1\leq H \leq H_1Z\leq \GL_n(q)$ where $H_1$ is almost quasisimple and $q=p^a$ for some integer $a$ and prime $p\geq 7$. If $F^*(H_1)/Z(F^*(H_1))\cong A_5$, then $H$ contains an involution $g$ such that 
 \[
  |H:C_H(g)|_{p'}\leq q^{n-1}+\cdots+q+1.
 \]
\end{lemma}
\begin{proof}
Suppose, first, that $F^*(H_1)$ is not simple. Since the Schur multiplier of $A_5$ has order $2$, this implies that $F^*(H_1)$ contains a unique central involution and the result follows. 

Suppose, then, that $F^*(H_1)=A_5$. Since $p\geq 7$ does not divide $|A_5|$, the representation $A_5\to \GL_n(q)$ decomposes into a direct sum of irreducibles. Consulting \cite{atlas} we see that the minimal degree of a non-trivial irreducible is $3$ and so $n\geq 3$. Now the result follows from the fact that $A_5$ contains a unique conjugacy class of involutions, and this class has order $15$.
\end{proof}

The next lemma completes the proof of Lemma A for subgroups from case 4.

\begin{lemma}\label{l: same char}
 Let $p\geq 7, q=p^f$ and $S$ be a quasisimple group of Lie type defined over a field of size $p^e$ lying inside $\GL(n,q)$, such that the associated vector space $V$ is an absolutely irreducible $\Fq S$-module. Then
$$|\Aut(S):C_{\Aut(S)}(g)|_{p'} \leq q^{n-1}+\dots + q+1.$$ 
\end{lemma}
\begin{proof}
We go through the possible quasisimple groups, as listed in the Classification of Finite Simple Groups. Since $p\geq 7$, $S$ is not isomorphic to ${^2B_2}(q)$, ${^2G_q}(q)$ or ${^2F_4}(q)$. Note that the lemma is trivial if $S$ contains a central involution; hence we assume that this is not the case. 

We assume that the representation $S\to \GL_n(q)$ induced by the given embedding cannot be realised over a proper subfield of $\mathbb{F}_q$, otherwise we redefine $q$ to take a smaller value. Since $p\geq 7$ we observe that the universal quasisimple cover of $S$ is a simply connected group of Lie type, hence we can apply \cite[Proposition 5.4.6]{kl}. 

Suppose $S$ is a quasisimple cover of $\PSL_m(p^e)$. If $m$ is odd, then $S$ contains an involution such that
$$|\Aut(S): C_{\Aut(S)}(g)|_{p'} = ((p^e)^{m-1}+\cdots + (p^e)+1.$$
Then \cite[Propositions 5.4.6 and 5.4.11]{kl} imply that $f|e$ and $n\geq \frac{em}f$, and the result follows.

If $m$ is even, then the absence of a central involution implies that $V$ is not the natural module for $S$. Then \cite[Propositions 5.4.6 and 5.4.11]{kl} imply that $n\geq \frac{e}{2f}m(m-1)$. We know that $S$ contains an involution such that
$$|\Aut(S): C_{\Aut(S)}(g)|_{p'} = ((p^e)^{m-2}+\cdots + (p^e)+1)((p^e)^{m-2}+\cdots + (p^e)^2+1).$$
Since $f|e$, the result follows immediately for $m\geq 4$. When $m=2$, $S=\PSL_2(q)$ and contains an involution $g$ such that
$$|\Aut(S): C_{\Aut(S)}(g)|_{p'} \leq p^e + 1.$$
Since $n\geq 2e/f$ we are done.

Suppose $S$ is a quasisimple cover of $\PSU_m(p^e)$ (so $S$ is defined over $\Fiel_{p^{2e}}$). Assume that $m>2$, since $\PSU_2(p^e)\cong \PSL_2(p^e)$ and this case is done. If $m$ is odd, then $S$ contains an involution such that
$$|\Aut(S): C_{\Aut(S)}(g)|_{p'} = (p^e)^{m-1}-\cdots -(p^e)+1.$$
Then \cite[Propositions 5.4.6 and 5.4.11]{kl} imply that $n\geq \frac{em}f$, and the result follows.

If $m$ is even, then, once again, the absence of a central involution implies that $V$ is not the natural module for $S$. Then \cite[Propositions 5.4.6 and 5.4.11]{kl} imply that $n\geq \frac{e}{2f}m(m-1)$. We know that $S$ contains an involution such that
$$|\Aut(S): C_{\Aut(S)}(g)|_{p'} = ((p^e)^{m-2}-\cdots -(p^e)+1)((p^e)^{m-2}+\cdots + (p^e)^2+1).$$
The result follows immediately for $m\geq 4$.

For the remaining classical groups, since $S$ has no central involution, we have $S$ simple. We take $g$ to be an involution that is central in the Levi subgroup $L$ of a parabolic subgroup of $S$. 

Suppose that $S=\PSp_{2m}(p^e)$ and assume that $m\geq 2$, since $\PSp_2(p^e)\cong \PSL_2(p^e)$. Choose $L$ to be isomorphic to $(q-1).(\PGL_2(q)\times \PSp_{n-2}(q))$ \cite[Proposition 4.1.19]{kl}. Then $L$ contains a central involution $g$ such that
$$|\Aut(S): C_{\Aut(S)}(g)|_{p'} \leq (p^e)^{2m-2}+\cdots +(p^e)^2+1.$$
Since $n\geq\frac{2me}{f}$, we are done.

Suppose that $S=\POmega^\epsilon_m(p^e)$. Assume that $m\geq 7$ as, otherwise, $P\Omega^\epsilon(p^e)$ is either not simple, or else is isomorphic to a simple group that we have already covered. Choose $L$ to be isomorphic to a Levi complement stabilizing a 2-dimensional totally singular subspace of the associated natural vector space, as described in \cite[Proposition 4.1.20]{kl}. If $m$ is odd, then $L$ contains a central involution $g$ such that
$$|\Aut(S): C_{\Aut(S)}(g)|_{p'} \leq (p^e)^{m-1}-1.$$
Since $n\geq\frac{me}{f}$, we are done.
If $m$ is even, then $L$ contains a central involution $g$ such that
$$|\Aut(S): C_{\Aut(S)}(g)|_{p'} \leq ((p^e)^{m/2-1}-1)((p^e)^{(m-2)/{2}}+1).$$
Since $n\geq\frac{me}{f}$, we are done.

Now for the exceptional groups. In what follows we use the name of the exceptional group to refer to the adjoint version (which is simple, since $p\geq 7$). 

Suppose first that $S=G_2(p^e)$; then \cite[Table 4.5.1]{gorenstein} implies that $S$ contains an involution $g$ such that $C_S(g)\geq \SL_2(p^e)\circ \SL_2(p^e)$. Thus
$$|\Aut(S): C_{\Aut(S)}(g)|_{p'} \leq e((p^e)^{4}+ (p^e)^2+1).$$
Now $f|e$ and $n\geq \frac{7e}{f}$ \cite[Propositions 5.4.6 and 5.4.12]{kl} and we are done.

Suppose that $S=F_4(p^e)$; then \cite[Table 4.5.1]{gorenstein} implies that $S$ contains an involution $g$ such that $C_S(g)\geq Spin_9(p^e)$. Thus
$$|\Aut(S): C_{\Aut(S)}(g)|_{p'} \leq e((p^e)^{8}+ (p^e)^4+1).$$
Now $f|e$ and $n\geq \frac{26e}{f}$ \cite[Propositions 5.4.6 and 5.4.12]{kl} and we are done.

Suppose that $S$ is a quasisimple cover of $E_6(p^e)$ (resp. ${^2E_6}(p^e)$); then \cite[Table 4.5.1]{gorenstein} implies that $S$ contains an involution $g$ such that $C_S(g)\geq \mathrm{Spin}_{10}^+(p^e)$ (resp. $\mathrm{Spin}_{10}^-(p^e)$). Thus $|\Aut(S): C_{\Aut(S)}(g)|_{p'}$ is at most
$$6e((p^e)^{8}+ (p^e)^4+1)((p^e)^{6}+ \epsilon(p^e)^3+1)((p^e)^{2}+ \epsilon(p^e)+1).$$
Now $f|2e$ and $n\geq \frac{27e}{f}$ \cite[Propositions 5.4.6 and 5.4.12]{kl} and we are done.

Suppose that $S=E_7(p^e)$; then \cite[Table 4.5.1]{gorenstein} implies that $S$ contains an involution $g$ such that $C_S(g)\geq \SL_8(p^e)/(4,q-1)$ or $\SU_8(q)/(4,q+1)$. Thus $|\Aut(S): C_{\Aut(S)}(g)|_{p'}$ is at most
$$2e((p^e)^{8}+ (p^e)^4+1)((p^e)^{12}+ (p^e)^6+1)((p^e)^{7}+ 1)((p^e)^{5}+1)((p^e)^{3}+1).$$
Now $f|e$ and $n\geq \frac{56e}{f}$ \cite[Propositions 5.4.6 and 5.4.12]{kl} and we are done.

Suppose that $S=E_8(p^e)$; then \cite[Table 4.5.1]{gorenstein} implies that $S$ contains an involution $g$ such that $C_S(g)\geq 2.(\PSL_2(p^e)\times E_7(p^e))$. Thus $|\Aut(S): C_{\Aut(S)}(g)|_{p'}$ is at most
$$e((p^e)^{10}+ 1)((p^e)^{12}+ 1)((p^e)^{6}+ 1)((p^e)^{30}-1).$$
Now $f|e$ and $n\geq \frac{248e}{f}$ \cite[Propositions 5.4.6 and 5.4.12]{kl} and we are done.

Finally suppose that $S={^3D_4}(p^e)$; then \cite[Table 4.5.1]{gorenstein} implies that $S$ contains an involution $g$ such that $C_S(g)\geq (\SL_2(p^{3e})\circ \SL_2(p^e)).2$. Thus
$$|\Aut(S): C_{\Aut(S)}(g)|_{p'} \leq 3e((p^e)^8+(p^e)^4+1).$$
Now either $f|e$ and $n\geq \frac{24e}{f}$ \cite[Comments after Proposition 5.4.6; Proposition 5.4.12]{kl}, or else $f|3e$ and $n\geq\frac{24e}f$ \cite[Comments after Proposition 5.4.6; Proposition 5.4.13]{kl} and we are done.
\end{proof}

\subsection{Odd-order subgroups of $\GL_n(q)$}

We are left with Case 2 of Corollary~\ref{c: breakdown}. In order to deal with this we must prove some facts about odd order subgroups of $\GL_n(q)$, where $q=p^f$ and $p\geq 7$.

We begin by quoting a result of P\'alfy \cite{palfy}.

\begin{lemma}\label{l: oddsn}
Let $H$ be a primitive subgroup of $S_n$ the symmetric group on $n$ letters. If $n>1$ and $H$ is solvable, then $|H|\leq 24^{-1/2}n^c$ where $c=3.24399\cdots$.
\end{lemma}

It is convenient to record a (basically weaker) variation of P\'alfy's result as follows.

\begin{lemma}\label{l: oddsn2}
 Let $H$ be a primitive subgroup of $S_n$, the symmetric group on $n$ letters. If $H$ has odd order, then $|H|\leq 2^{n-1}$.
\end{lemma}
\begin{proof}
 Since $|H|$ is odd, $H$ is soluble, and the bound in Lemma~\ref{l: oddsn} applies. This yields the result for $n\geq 10$. For $n\leq 9$ we refer to \cite{atlas} to confirm the result in all cases.
\end{proof}

The next couple of results are aimed at giving an upper bound for $|H|_{p'}$ where $H$ is an odd-order subgroup of $\GL_n(q)$. The background idea here is that a prototypical `large' odd-order subgroup $H<\GL_n(q)$ contains a torus as a large normal subgroup. More precisely, we think of $H$ as equalling a direct product $H_1\times H_2 \times \cdots H_k$ (for some $k\leq n$) where, for each $i$, the group $H_i$ acts semi-linearly on some $n_i$-dimensional vector space over the field $\mathbb{F}_{q^{a_i}}$ (and $n=\sum_{i=1}^k a_in_i$). What is more $H_i$ normalizes a maximal split torus $T_i$ of the group $\GL_{n_i}(q^{a_i})$. (i.e. $H_i$ preserves a decomposition of the associated $n_i$-dimensional vector space over $\mathbb{F}_{q^{a_i}}$ into $1$-dimensional subspaces).

We know that $|T_i|<q^{a_in_i}$; we must study the size of $|H_i:H_i\cap T_i|$. In the action of $H_i$ on the aforementioned decomposition, $H_i$ can induce field automorphisms of order at most $a_i$ on each $1$-dimensional subspace - this allows for a possible index of $a_i^{n_i}$. In addition, $H_i$ can permute the $1$-dimensional subspaces. We can reduce to the case where the action on the set of subspaces is primitive, in which case Lemma~\ref{l: oddsn2} allows for a possible index of $2^{n_i-1}$. This is more than cancelled out by the fact that $|T_i: H_i\cap T_i|\geq 2^{n_i}$. This reasoning motivates the following definition.

For a fixed positive integer $n$, we define a {\it partition of $n$} to be a positive list of integers, $\lambda$, that sum to $n$; i.e.
\begin{equation}\label{lambda}
\lambda = [\underbrace{1,\dots,1}_{n_1},\underbrace{2,\dots,2}_{n_2}, \underbrace{3,\dots,3}_{n_3},\dots ]
\end{equation}
where $n=\sum_i in_i$. For a partition $\lambda$ of this form define 
\[
\begin{aligned}
\ell_\lambda&=|\{i \mid n_i > 0 \}|, \textrm{ and }
c_\lambda&=\frac{|1^{n_1}2^{n_2}3^{n_3}\cdots|_{2'}}{2^{\ell_\lambda}}. 
\end{aligned}
\]
When we write $1^{n_1}2^{n_2}3^{n_3}\cdots$ here, we omit all terms $i^{n_i}$ for which $n_i=0$; this leaves a finite number of terms in our definition. Now define 
\[c_n = \max\{c_\lambda \mid \lambda \textrm{ is a partition of } n\}.\]

\begin{lemma}\label{l: cn}
\mbox{}
\begin{enumerate}
\item If $n=mk$, for $1<m,k\in \mathbb{Z}^+$, then $c_n\geq (c_m)^k2^{k-1}$;
\item If $n=r+s$, for $0<r,s\in\mathbb{Z}^+$, then $c_n\geq c_{r}c_{s}$;
\item Suppose that $\lambda$ is a partition of $n$ for which $c_n=c_\lambda$. \begin{itemize}                                                                               \item If $n\equiv 0\pmod 3$, then $\lambda = [3,\dots,3]$ and
$c_n =\frac12 3^{\frac{n}{3}}.$
\item If $n\equiv 1\pmod 3$, then $\lambda = [1,3,\dots,3]$ and
$c_n =\frac14 3^{\frac{n-1}{3}}.$
\item If $2<n\equiv 2\pmod 3$, then $\lambda = [3,\dots,3,5]$ and
$c_n =\frac54 3^{\frac{n-2}{3}}.$
\item If $n=2$, then $\lambda=[1,1]$ or $[2]$ and
$c_n= \frac 12$.
 \end{itemize}
In particular
\[\frac14 3^{\frac{n-1}{3}}\leq c_n\leq  \frac12 3^{n/3} < 7^{\min\{n/5, n/2-1\}}.
\]
\end{enumerate}
 \end{lemma}
\begin{proof}

Let $n=mk$ and let 
\[\lambda_m=[\underbrace{1,\dots,1}_{m_1},\underbrace{2,\dots,2}_{m_2}, \underbrace{3,\dots,3}_{m_3},\dots ]\]
be a partition of $m$ for which $c_m = c_{\lambda_m}$ and set $\ell:=\ell_{\lambda_m}$. Now let $\lambda$ be the partition of $n$ achieved by joining $k$ copies of $\lambda_m$ together. We have
\[
 \begin{aligned}
  c_{\lambda_m}&=\frac{|1^{m_1}2^{m_2}3^{m_3}\cdots|_{2'}}{2^{\ell}}, \textrm{ and } \\
  c_{\lambda}&=\frac{|1^{km_1}2^{km_2}3^{km_3}\cdots|_{2'}}{2^{\ell}}
  &=(c_{\lambda_m})^k\cdot 2^{\ell k - \ell}
 \end{aligned}
\]
and (1) follows.

For (2) suppose that $n=r+s$, and let $\lambda_r$ (resp. $\lambda_s$) be a partition of $r$ (resp. $s$) for which $c_r = c_{\lambda_r}$ (resp. $c_s=c_{\lambda_s}$). Now let $\lambda$ be the partition of $n$ achieved by joining $\lambda_r$ to $\lambda_s$; then $c_\lambda \geq c_{\lambda_r}c_{\lambda_s}$.

We are left with (3). Suppose that $\lambda$ is the partition of $n$ which is defined by \eqref{lambda} and for which $c_\lambda=c_n$.

Suppose that $1n_1+2n_2+4n_4\geq 3$. Consider the partition, $\lambda_0$, obtained by removing all occurrences of $1,2$ and $4$ and instead inserting as many $3$'s as possible, along with either zero, one or two $1$'s. Now compare $c_\lambda$ and $c_{\lambda_0}$: it is clear that the numerator of $c_{\lambda_0}$ is at least three times as large as the numerator of $c_{\lambda}$. On the other hand $\ell_{\lambda_0}\leq \ell_{\lambda}+1$. We conclude that $c_{\lambda_0}\geq c_{\lambda}$. This is a contradiction, hence we may assume that $1n_1+2n_2+4n_4<3$,

Now suppose that $n_k>0$ for some $k>4$ and set $m:=kn_k$. Consider the partition $\lambda_1$, obtained by removing all occurrences of $k$ and instead inserting as many $3$'s as possible, along with either zero, one or two $1$'s. We compare $c_\lambda$ and $c_{\lambda_1}$ as before. As before $\ell_{\lambda_1}\leq \ell_{\lambda}+1$. Thus $c_{\lambda_1}> c_{\lambda}$ unless
\[
 \begin{aligned}
  &\frac12 3^{\frac{m-2}{3}} \leq k^{m/k} \\
\Longleftrightarrow & \frac 12 \mathrm{e}^{\frac{(\log 3)(m-2)}{3}} \leq \mathrm{e}^{m(\frac{\log k}{k})} \\
\Longleftrightarrow & \frac{(\log 3)(m-2)}{3\log 2} < m(\frac{\log k}{k}) \\
\Longrightarrow & m<k \textrm{ or } m\leq 5.
 \end{aligned}
\]
Since $m\geq k$ by definition, we conclude that, unless $m=k=5$, $c_{\lambda_1}> c_{\lambda}$ which is a contradiction. Thus by repeating this step (and the previous if necessary) we conclude that $\lambda=1^{n_1}2^{n_2}3^{n_3}5^{n_5}$ where $n_1+2n_2\leq 2$ and $n_5\leq 1$.

Suppose, next, that $n_5=1$ and $n_1+n_2>0$. Consider the partition $\lambda_2$, obtained by removing the $5$ and one of the numbers less than $3$, and instead inserting two $3$'s, along with either zero or one $1$. We compare $c_{\lambda_2}$ and $c_\lambda$: it is clear that the numerator of $c_{\lambda_2}$ is strictly larger than that of $c_{\lambda}$, while the denominator is unchanged. This is a contradiction, and we conclude that if $n_5=1$, then $n_1+n_2=0$.

Suppose, finally, that $n_3>0$ and $n_1+2n_2=2$. Consider the partition $\lambda_3$, obtained by removing a $3$ and the numbers less than $3$, and instead inserting a $5$. We compare $c_{\lambda_3}$ and $c_\lambda$: it is clear that the numerator of $c_{\lambda_3}$ is strictly larger than that of $c_{\lambda}$, while the denominator is either smaller or the same. This is a contradiction, and so we conclude that if $n_3>0$ then $n_1+2n_2\leq 1$.

The result follows.
\end{proof}

\begin{lemma}\label{l: oddgln}
Suppose $H<GL_n(q)$ where $q=p^f$ and $p\geq 7$. If $H$ has odd order then $|H|_{p'}< q^n c_n$.
\end{lemma}
\begin{proof}
 The result is immediate for $n=1$. Lemma~\ref{l: psl2} implies that, for $n=2$, $|H|_{p'}<q+1$ and the result follows.  For $n>2$, we refer to Lemma \ref{l: relevant groups}. Note that $H$ cannot lie in Case 3 by the Feit-Thompson theorem; Case 1 is also impossible and the result is immediate in Case 5. If $H$ lies in Case 4 then $n=4$ and $H\leq GO_4^+(q)$; then $|H|_{p'}\leq \frac12(q-1)|H\cap\Omega_4^+(q)|_{p'}$. Now $\Omega_4^+(q)\cong SL_2(q)\circ SL_2(q)$ and, again, Lemma \ref{l: psl2} implies that $|H\cap\Omega_4^+(q)|_{p'}<(q+1)^2$. Thus $|H|_{p'}<\frac12(q-1)(q+1)^2$ and we are done. 

We are left with the possibility that $H$ lies in Case 2. We proceed by induction; the base case, $n=1$, is already done. Now assume inductively that the statement is true for $H<GL_m(q)$ where $m<n$. 

If $H$ lies in Case 5 then $H\leq GL_m(q)\circ GL_m(q)$. Induction implies that 
$$|H|_{p'}<(q^mc_m)^2 \leq q^{2m}(c_m)^2<q^{m^2}(c_m)^2.$$
Now Lemma~\ref{l: cn} (2) yields the result.

Now consider Case 2. First observe that if $H\leq M\in\curly_5(GL_n(q))$ then $M=GL(n,q_0)\circ Z(GL(n,q))$ where $q=q_0^a$. Then it is sufficient to prove the result for $H\cap \GL_n(q_0)$ considered as a subgroup of $\GL_n(q_0)$. Thus we assume that $H$ does not lie in a maximal subgroup of type 5.

If $H$ lies in a parabolic subgroup of $GL_n(q)$ then $H/O_p(H)<GL_m(q)\times GL_{n-m}(q)$. Induction implies that 
\[|H|_{p'}< q^mc_m\times c^{n-m}c_{n-m}.\]
Now Lemma~\ref{l: cn} (2) yields the result.

If $H\leq M\in\curly_4(GL_n(q))$ then $M=(GL_{r}(q)\circ GL_{s}(q))$ for $n=rs$ with $2\leq r<s$. Induction implies that
\[|H|_{p'}<q^{r}c_{r}q^{s}c_{s} = q^{r+s}c_{r}c_{s}.\]
Lemma~\ref{l: cn} (2) yields the result.

If $H\leq M\in\curly_2(GL_n(q))$ then $M=GL_m(q)\wr S_t$ where $n=mt, t\geq 2$. We assume that $H$ acts transitively on the $m$-space decomposition otherwise $H$ lies in a parabolic subgroup. In fact we assume that $H$ acts primitively on the $m$-space decomposition since otherwise $H$ lies in a maximal subgroup $M_1=GL_{m_1}(q)\wr S_{t_1}$, $M_1\in\curly_2(GL_n(q))$ with $t_1<t$ and $H$ acts primitively on this decomposition. 

Induction and Lemma \ref{l: oddsn2} imply that
\[|H|_{p'}<(q^mc_m)^t 2^{t-1}.\]
Lemma~\ref{l: cn} (1) yields the result.

If $H\leq M\in\curly_3(GL_n(q))$ then $M=GL_m(q^r).r$ where $n=mr$ and $r$ is a prime. Induction implies that
$$|H|_{p'}<((q^r)^m)c_m.r=q^nc_m|r|_{2'}.$$
If $r=2$, then it is enough to check that $c_m\leq c_{2m}$ which is easy using the formulae given in Lemma~\ref{l: cn} (3). If $r\geq 3$, then one can check the result directly using Lemma~\ref{l: cn} (3).
\end{proof}

It is worth combining results of the two previous lemmas to give the following immediate corollary that will be used repeatedly. (The corollary is a weak version of Proposition~\ref{p: oddgln} which was stated in the introduction.)

\begin{corollary}\label{c: oddgln}
Suppose $H$ is an odd order subgroup of $\GL_n(q)$ where $q=p^f, p\geq 7$ and $n\geq 2$. Then $|H|_{p'}< q^{\frac32n-1}$.
\end{corollary}

\begin{lemma}\label{L: second gap}
Let $H$ be an odd order subgroup of $\GL_n(q)$ and let $\sigma$ be an involutory field automorphism of $\GL_n(q)$. Suppose that $H$ is normalized by $g$, an involution in $\langle \GL_n(q), \sigma\rangle\backslash \GL_n(q)$. Then
$$|H:C_H(g)|_{p'}\leq q^{n-\frac12}+q^{n-1}+\dots+q+q^{\frac12}+1.$$
\end{lemma}
\begin{proof}
Note that, by \cite[Proposition 4.9.1]{gorenstein}, we know that $g$ is $\GL_n(q)$-conjugate to $\sigma$ or $(-I, \sigma)$, hence we take $H$ to be normalized by $\sigma$. Observe that $C_{Z(G)}(g)=(\sqrt{q}-1)$ (which yields the result for $n=1$) and that $q\geq 49$. 

For $n=2$, it is clear that $|H:C_H(g)|_{p'}< (\sqrt{q}+1)k$ where $k=\max\{|H_1|_{p'}: H_1\in \PSL_2(q)\}$; we refer to Lemma \ref{l: psl2}, and the statement follows immediately. 

For $n>2$ we refer to Corollary \ref{c: breakdown}, and observe that Cases 1 and 4 are impossible. If we are in Case 5 then $n=4$ and $H\leq GO_4^+(q)$. Now $\Omega_4^+(q)\cong SL_2(q)\circ SL_2(q)$; thus Lemma \ref{l: psl2} implies that $|H\cap \Omega_4^+(q)|_{p'}< (q+1)^2$ which yields the result immediately.

Suppose next that we are in Case 2 or 3. In particular suppose first that, either Case 3 holds, or $H\leq M\in\curly_i(\langle GL_n(q), \sigma\rangle)$ for $i=6,7$. Then Lemma \ref{l: relevant groups} implies that $M=(GL_n(q)\circ GL_n(q)).(\langle \sigma \rangle \times 2)$ and $H<N=GL_{\sqrt{n}}(q)\circ GL_{\sqrt{n}}(q)$. $N$ is normalized in $GL_n(q)$ by $\tau$, an involution which swaps the two copies of $GL_{\sqrt{n}}(q)$. Thus $g$ may take two forms.

Firstly suppose that $g=(A,B)\sigma$ where $A,B\in GL_{\sqrt{n}}(q)$. Now $N$ contains two normal subgroups isomorphic to $\GL_{\sqrt{n}}(q)$ that are normalized by $g$ and, by \cite[Proposition 1.1]{bgl}, there is only one class of involutions in each copy of $\langle \GL_{\sqrt{n}}(q), g\rangle$. Then, by induction,
$$|H:C_H(\sigma)|_{p'}\leq (q^{\sqrt{n}-\frac12}+\dots + q+q^{\frac12}+1)^2<q^{n-\frac12}+q^{n-1}+\dots+q+q^{\frac12}+1.$$

Secondly suppose that $g=(A,B)\tau\sigma$. Since $g$ is an involution, $B=\pm A^{-\sigma}$. Then, for $(X,Y)\in N$, 
$$(X,Y)^g = (AY^\sigma A^{-1}, A^{-\sigma}X^\sigma A^\sigma).$$
Thus $(X,Y)$ will be centralized by $g$ if and only if $X=AY^\sigma A^{-1}$. Thus $|N:C_N(g)|=|\GL_{\sqrt{n}}(q)|$ and so
\[|H:C_H(g)|_{p'}\leq |\GL_{\sqrt{n}}(q)|_{p'} < q^{\frac{\sqrt{n}(\sqrt{n}+1)}2}
 \]
and the result follows.

We are left with the possibility that $H\leq M\in\curly_i(\langle \GL_n(q), \sigma\rangle)$ for $i=1,\dots, 5$. In this case we proceed by induction; the base case has already been attended to.

If $M\in\curly_5(\langle \GL_n(q), \sigma\rangle)$, then $|H|_{p'}<q|H_0|_{p'}$, where $H_0$ is an odd order subgroup of $GL_n(q_0)$ where $q=q_0^a$ for some integer $a>1$. Then Corollary~\ref{c: oddgln} implies that
$$|H|_{p'}\leq q|H_0|_{p'}<q^{\frac{n+\sqrt{n}}2+1}$$
and the result follows.

If $M\in\curly_1(\langle \GL_n(q), \sigma\rangle)$, then let $O_p(M)$ be the largest normal $p$-group of $M$ and let $L=H/(H\cap O_p(M))$. Then
$$|H:C_H(g)|_{p'} = |L:C_L(g_L)|_{p'}$$
for some involution $g_L\in L$. Now $L\leq \GL_m(q)\times \GL_{n-m}(q)$ and we apply induction:
\begin{equation*}
\begin{aligned}
|H:C_H(g)|_{p'}&< (q^{m-\frac12}+\dots+q^{\frac12}+1)(q^{n-m-\frac12}+\dots+q^{\frac12}+1) \\
&<q^{n-\frac12}+q^{n-1}+\dots+q+q^{\frac12}+1. 
\end{aligned}
\end{equation*}

A similar approach can be taken if $M\in\curly_4(\langle \GL_n(q), \sigma\rangle)$. Then $M\cong (\GL_m(q)\circ \GL_{t}(q))\langle\sigma\rangle$ for $n=mt$ and, once more, induction gives the result.

Now suppose that $M\in\curly_3(\langle \GL_n(q), \sigma\rangle)$, so $M\cong (\GL_{\frac{n}{r}}(q^r).r)\langle\sigma\rangle$ with $r$ prime. If $r=2$, then any element from $\langle \GL_n(q),\sigma\rangle \backslash \GL_n(q)$ which normalizes $M$ will act as a field automorphism of order $4$ on $\GL_{\frac{n}{r}}(q^r)$. In particular such an element cannot be an involution. On the other hand if $r$ is odd then $\langle M, \sigma\rangle\cong \GL_{\frac{n}{r}}(q^r).2r$ and $g$ acts as an involutory field automorphism on $\GL_{\frac{n}{r}}(q^r)$. Then Lemma \ref{L: oddnormal} implies that
$$|H:C_H(g)|_{p'} = |H\cap \GL_{\frac{n}{r}}(q^r):C_{H\cap \GL_{\frac{n}{r}}(q^r)}(g)|_{p'}$$
and induction gives the result.

Finally consider the possibility that $M\in\curly_2(\langle \GL_n(q), \sigma\rangle)$. Thus $H<\langle(\GL_m(q)\wr S_t),\sigma\rangle$ with $t\geq 2$. Just as for Lemma \ref{l: oddgln}, we assume that $\langle H,\sigma\rangle$ acts primitively on the $m$-space decomposition. Take $g=s\sigma$ and note that $s$ acts as a (possibly trivial) involution on the $m$-space decomposition.\label{ten} 

We need to consider two situations which closely mirror the two cases discussed for $\curly_7$. First consider $C_{S}(g)$ where $S$ is the projection of $H$ onto a particular $\GL_m(q)$ that is fixed by $s$. By induction $|S: C_{S}(g)|_{p'}<q^{m-\frac12}+\dots+q^{\frac12}+1.$ Alternatively if $\GL_m(q)\times \GL_m(q)$ are swapped by $s$, and $S$ is the projection of $H$ onto $\GL_m(q)\times \GL_m(q)$ then it is clear that $|S: C_{S}(g)|_{p'}$ is at most the size of an odd-order subgroup in $\GL_m(q)$. Thus, by Corollary~\ref{c: oddgln}, this is bounded above by $q^{\frac{3m-2}{2}}$.

Now write $N=\underbrace{\GL_m(q)\times\dots\times \GL_m(q)}_t$. Then Lemma \ref{L: oddnormal} implies that 
$$|H:C_H(g)|=|H\cap N: C_{H\cap N}(g)||H/N: C_{H/N}(gN)|.$$
Thus, writing $s$ as the product of $k$ transpositions in its action on the $m$-space decomposition, we have
\begin{eqnarray*}
|H\cap N: C_{H\cap N}(g)|_{p'}&<&(q^{m-\frac12}+\dots+q+1)^{t-2k}\times (q^{\frac{3m-2}{2}})^k \\
&<& \frac{(q^m-1)^t}{\sqrt{q}-1} \times \frac{1}{(\sqrt{q}-1)^{t-1}}.
\end{eqnarray*}
If $t\geq 2$ then, referring to Lemma \ref{l: oddsn}, it is sufficient to prove that $3.25^t<(\sqrt{q}-1)^{t-1}$ which is true, since $\sqrt{q}\geq 7$. If $t=2$ the result follows from the fact that
\[
 |H:C_H(g)|_{p'}=|H\cap N: C_{H\cap N}(g)|_{p'}.
\]
\end{proof}

\begin{lemma}\label{L: reduced bound}
Let $H<GL_n(q)$ with $|H|$ odd. Suppose that $g$ is an involution in $GL_n(q)$ which normalizes $H$. Then,
$$|H:C_H(g)|_{p'}\leq q^{n-1}+\dots+q+1.$$
\end{lemma}
\begin{proof}
We proceed by induction on $n$. When $n=1$ the result is trivial and, when $n=2$, Lemma \ref{l: gl2} gives the result.

Now suppose that $n>2$ and refer to Lemma \ref{l: relevant groups} for the group $\langle H,g\rangle$. Clearly Cases 1 and 3 are not relevant here. If Case 4 or Case 5 holds then our result is implied by the strong version of Lemma A which we have already proved in these cases. In Case 6 the result is immediate. Thus we are left with Case 2: $\langle H,g\rangle \leq M\in \curly_i(GL_n(q))$ for $i=1,\dots, 5$.

If $M\in \curly_5(GL_n(q))$ then we can assume that $\langle H,g\rangle \leq GL_n(q_0)$ where $q=q_0^a$. Then it is sufficient to prove the result over the base case. Thus we assume that $M\not\in\curly_5(GL_n(q))$. 

If $\langle H,g\rangle \leq M\in\curly_1(\GL_n(q))$ then $\langle H,g\rangle \leq Q:(\GL_m(q)\times \GL_{n-m}(q))$ where $Q=O_p(M)$ and $m>1$. Set $L=M/Q$ and observe that
$$|H:C_H(g)|_{p'} = |H_L: C_{H_L}(g_L)|_{p'}$$
where $H_L=HQ/Q\leq L$ and $g_L=gQ\in H_L$. Now suppose, without loss of generality, that $m\leq n-m$, and set $N=H_L\cap \GL_m(q)$. 

Lemma \ref{L: oddnormal} implies that
$$|H_L: C_{H_L}(g_L)| = |N:C_N(g)| \times |H/N:C_{H/N}(g_LN)|.$$
Now $H/N$ is isomorphic to a subgroup of $\GL_{n-m}(q)$ and induction implies that 
$$|H/N:C_{H/N}(g_LN)|_{p'} \leq q^{n-m-1}+\dots+q+1.$$
Then $g_L$ normalizes $N$ and $g_L$ either centralizes $N$ or else $\langle g_L, N\rangle$ is isomorphic to a subgroup of $\GL_m(q)$. Either way induction implies that
$$|H_L:C_{H_L}(g_L)|_{p'} \leq q^{m-1}+\dots+q+1,$$
as required.

If $\langle H,g\rangle \leq M\in\curly_3(\GL_n(q))$ then $\langle H,g\rangle\leq M=\GL_m(q^r).r$ where $n=mr$ and $r$ is prime. Let $N=\GL_m(q^r)$ be normal in $M$ and split into two cases. Suppose first that $|\langle H,g\rangle \cap N|$ is even. Then induction implies that
$$|H:C_H(g)|_{p'}\leq ((q^r)^{m-1}+\dots+q+1)r< q^{mr-1}+\dots+q+1$$
as required. Suppose on the other hand that $|H\cap N|$ is odd. Then $r=2$ and Lemma \ref{L: second gap} gives the result.

If $\langle H,g\rangle \leq M\in\curly_4(\GL_n(q))$ then $M\cong \GL_{n_1}(q)\circ GL_{n_2}(q)$ where $n=n_1n_2$ and $1<n_1<n_2$. Proceed similarly to the case $\curly_1$, observing that $M$ has normal subgroups isomorphic to both $\GL_{n_1}(q)$ and $\GL_{n_2}(q)$; thus we define $N=H\cap \GL_{n_1}(q)$. Suppose first that $N$ has even order and take $g$ to be an involution in $N$. Then Lemma \ref{L: sylowtwos} implies that, if $P$ is a Sylow $2$-subgroup of $N$< then
$$|H:C_H(g)| = |N:C_N(g)|\frac{|g^H\cap P|}{|g^N\cap P|}=|N:C_N(g)|.$$
By induction 
$$|N:C_N(g)|_{p'}\leq q^{n_1-1}+\cdots+q+1$$
and we are done. Suppose, on the other hand, that $N$ has odd order. In particular this means that $H\cap Z(M)$ has order $z$, an odd number. Now Lemma \ref{L: oddnormal} implies that, for $g$ an involution in $H$,
$$|H: C_{H}(g)| = |N:C_N(g)| \times |H/N:C_{H/N}(gN)|.$$
Observe that $H/N$ is a subgroup of $\PGL_{n_2}(q)$. Let $H_X$ be the subgroup of $\GL_{n_1}(q)\times \GL_{n_2}(q)$ such that $H_X\cap Z(GL_{n_1}(q)\times \GL_{n_2}(q))\cong z\times z$, and such that $\pi(H_X)=H$ where $\pi$ is the natural map,
$$\pi:\GL_{n_1}(q)\times \GL_{n_2}(q)\to \GL_{n_1}(q)\circ \GL_{n_2}(q).$$
Let $g_X\in H_X$ be the unique involution for which $\pi(g_X)=g$. Let $g_2$ (resp. $H_2$) be the image of $g_X$ (resp. $H_X$) under the natural projection $\GL_{n_1}(q)\times \GL_{n_2}(q)\to GL_{n_2}(q)$. Observe that $H_2/(H_2\cap Z(\GL_{n_2}(q))\cong H/N$. By induction
\[
 \begin{aligned}
|H_2: C_{H_2}(g_2)|_{p'} & \leq q^{n_2-1}+\dots+q+1, \\
  |N:C_N(g)|_{p'} & \leq q^{n_1-1}+\cdots+q+1
 \end{aligned}
\]
and the result follows.

Finally if $M\in \curly_2(GL_n(q))$ then $\langle H, g\rangle <GL_m(q)\wr S_t$ where $n=mt$. Set $N=\underbrace{GL_m(q)\times\dots\times GL_m(q)}_t$. We assume that $N/H$ acts primitively on the $t$ copies of $GL_m(q)$ as otherwise $H$ lies in a maximal subgroup $M_1=GL_{m_1}(q)\wr S_{t_1}$, $M_1\in\curly_2(GL_n(q))$ with $t_1<t$ such that $H$ acts primitively on this decomposition.

If $g\in N$ then induction implies that $|H\cap N:C_{H\cap N}(g)|_{p'}<(\frac{q^m-1}{q-1})^t$. Lemma \ref{l: oddsn} implies that $|H/N|<\frac{1}{\sqrt{24}}t^{3.25}$ which is sufficient unless $t=2$, when the result is immediate.

If $g\in H\backslash N$, then we proceed very similarly to the proof of Lemma \ref{L: second gap}. First consider $C_{S}(g)$ where $S$ is the projection of $H$ onto a particular $GL_m(q)$ that is fixed by $g$. By induction $|S: C_{S}(g)|_{p'}<q^{m-1}+\dots+q+1$. Alternatively if $GL_m(q)\times GL_m(q)$ are swapped by $g$, and $S$ is the projection of $H$ onto $GL_m(q)\times GL_m(q)$ then it is clear that $|S: C_{S}(g)|_{p'}$ is at most the size of an odd-order subgroup in $GL_m(q)$. 

If $m>1$, then, by Corollary \ref{c: oddgln}, this is bounded above by $q^{\frac{3m-2}{2}}$. Thus, for $m>1$, $|H\cap N:C_{H\cap N}(g)|_{p',}<(q^{m-1}+\dots+q+1)^t$. As before, Lemma \ref{l: oddsn} yields the result.

If $g\in H\backslash N$ and $m=1$ then it is clear that $|C_{H\cap N}(g)|>(q-1)^{\lceil\frac{n}2\rceil}$ and so $|H\cap N: C_{H\cap N}(g)|_{p'}\leq (\frac{q-1}2)^{\frac{n}2}$. It is therefore sufficient to prove that $|H/N|_{p'}<(2(q-1))^{\frac{n}2-1}$. Again Lemma~\ref{l: oddsn} implies that $|H/N|\leq \frac{1}{\sqrt{24}}n^{3.25}$ and the result follows for $n\geq 6$. If $n=5$ (resp. $n\leq 4$), then $|H/N|_{p'}\leq 5$ (resp. $|H/N|_{p'}\leq 3$) and the result follows immediately.
\end{proof}

\begin{lemma}\label{L: klein}
Let $H<GL_n(q)$ with $|H|$ odd. Let $\sigma$ be a field automorphism of $GL_n(q)$ satisfying $\sigma^2=1$. Suppose that $K$ is a 2-group in $\langle GL_n(q), \sigma\rangle$ which normalizes $H$. Then
$$|H:C_H(g)|_{p'}\cdot|H:C_H(h)|_{p'}<q^{\frac{3n}2}c_n$$
for distinct involutions $g,h\in K$.
\end{lemma}

Note that if $q$ is not square then $\langle GL_n(q),\sigma\rangle = GL_n(q)$. Note too that $c_n$ was defined on p.~\pageref{l: cn} immediately before Lemma \ref{l: cn}. 

\begin{proof}
If $n=1$, the result is immediate. If $n=2$, then we must show that
$$|H:C_H(g)|_{p'}\cdot |H:C_H(h)|_{p'}<\frac12q^3.$$
Lemma \ref{l: gl2} implies that there exists $g$ such that
$$|H:C_H(g)|_{p'}<q+1.$$
Now Lemma \ref{l: psl2} implies that $|H|_{p'}<\frac14(q-1)^2$ and the result follows.

Now take $n\geq 3$ and consider $\langle H, K\rangle$ as a subgroup of $\langle GL_n(q), \sigma\rangle$ and refer to Corollary \ref{c: breakdown}. We consider the five cases listed there; observe first that Case 4 is impossible. Consider Case 1 next: $\langle H, K\rangle$ contains $\Omega$. Then, since $\langle H,K\rangle$ is soluble this implies that $n=1$ which has been covered.

Now consider Case 5: we lie in the small rank setting. We have dealt with $n=2$ hence we are left with $n=4$ and $H<GO_4^+(q)\cong Z(SL_2(q)\circ SL_2(q)).4$; then $|H|_{p'}\leq \frac1{16}(q-1)(q+1)^2$ and we are done.

We move on to Case 3 which we expand to cover $H\leq M\in\curly_6(\Delta_\eta)$ for $\eta$ non-degenerate or zero. Now it is just a matter of checking that $(|M|_{p'})^2 < c_nq^n$ and the result follows immediately.

We are left with Case 2 or, more precisely, with the possibility that $\langle H, K\rangle\leq M_1\in\curly_i(\langle GL_n(q), \sigma \rangle)$ for some $i=1,\dots, 5, 7$. Then $H\leq M\in\curly_i(GL_n(q))$ for some $i=1,\dots, 5, 7$. Clearly if $M\in\curly_5(GL_n(q))$ then it is sufficient to prove the result for $H\cap M$ taking the place of $H$. 

If $M\in\curly_7(GL_n(q)$ then $M\cong (GL_m(q)\circ\cdots \circ GL_m(q)).S_t$ with $n=m^t$ and $m\geq 3$. We can assume, as usual, that the action of $S_t$ in the wreath product is primitive and now Lemma \ref{l: oddsn2} and Corollary \ref{c: oddgln} imply that 
$$|H|_{p'}< (q^{\frac{3m-2}2})^t2^{t-1}.$$
This yields the result immediately.

We are left with the possibility that $H\leq M\in\curly_i(GL_n(q))$ for some $i=1,\dots,4$ and we proceed by induction. If $M\in\curly_1(GL_n(q)$ or $M\in\curly_4(GL_n(q)$, then the result is immediate using induction and Lemma~\ref{L: oddnormal}, since in this case $M/O_p(M)$ is isomorphic to $GL_m(q)\times GL_{n-m}(q)$ or $GL_m(q)\circ GL_{\frac{n}m}(q)$. 

If $M\in\curly_3(GL_n(q)$, then we have $M\cong GL_{\frac{n}r}(q^r).r$ where $r$ is prime. If $r=2$ then no involutory field automorphism of $GL_n(q)$ normalizes $M$ hence we $\langle H,K\rangle$ is isomorphic to a subgroup of $GL_{\frac{n}r}(q^r).2$ and the result follows by induction. 

Suppose, then, that $r$ is odd and set $m:=\frac{n}{r}$ and $H_1:=H\cap GL_m(q^r)$. Observe that $K\leq H_1$. Consider the case when $m=1$. Then $M$ does not contain a Klein 4-group, so we have a contradiction. If $m=2$, then $M\cong \GL_2(q^r).r$ Note that any Klein 4-subgroup of $\GL_2(q^r)$ intersects $Z(\GL_2(q^r)$ non-trivially. Thus Lemma~\ref{L: reduced bound} implies that
\[|H:C_H(g)|_{p'}\cdot |H:C_H(h)|_{p'} \leq r^2(q^{n-1}+\cdots +q+1)\]
and the result follows.

If $m\geq 2$, then we apply induction to conclude that
\[|H_1:C_{H_1}(g)|_{p'}\cdot |H_1:C_{H_1}(h)|_{p'}<q^{\frac{3n}2}c_{m}\]
and so 
\[|H:C_H(g)|_{p'}\cdot |H:C_H(h)|_{p'}< q^{\frac{3n}2}c_{m}r^2.\]
Now one can check, using the formulae given at Lemma~\ref{l: cn} (3), that $c_n>c_mr^2$ and the result follows.

Finally assume that $H\leq M\in\curly_2(GL_n(q))$. Then $M\cong GL_m(q)\wr S_t$. Write $N$ for the normal subgroup of $M$ isomorphic to $\underbrace{GL_m(q)\times \cdots \times GL_m(q)}_t$ and let $H_1 = \langle H,K\rangle \cap N$; we may assume as usual that $\langle H,K\rangle/H_1$ acts primitively in the natural action on $t$ copies of $GL_m(q)$. 

Consider the action of $K=\langle g, h\rangle$ on the $t$ copies of $GL_m(q)$; suppose that $K$ has an orbit of length $t_1>1$ (note that $t_1$ is even). Then $K$ normalizes $H_2 = H\cap \underbrace{(GL_m(q)\times\cdots\times GL_m(q))}_{t_1}$ and we consider the induced action of $K_1$ on a set of size $t_1$; we can assume that $g$ is a product of $\frac{t_1}2$ transpositions, while $h$ is a product of either $\frac{t_1}2-1$ or $\frac{t_1}2$ transpositions. Lemma \ref{l: oddgln} implies that 
$$|H_2:C_{H_2}(g)|_{p'} \leq (q^mc_m)^{\frac{t_1}2};$$
similarly Lemmas \ref{l: oddgln}, \ref{L: second gap} and \ref{L: reduced bound} imply that
$$|H_2:C_{H_2}(h)|_{p'} \leq (q^mc_m)^{\frac{t_1}2-1}q^{2m}$$
which implies, in particular, that
\[|H_2:C_{H_2}(g)|_{p'}\cdot |H_2:C_{H_2}(h)|_{p'} \leq (q^{\frac{3m}{2}})^{t_1}c_m^{t_1}.\]
If $K_1$ has an orbit of length $1$ then $K_1$ normalizes $H_3 = H\cap GL_m(q)$, for a particular copy of $GL_m(q)$. We apply induction to conclude that
$$|H_3:C_{H_3}(g)|_{p'}\cdot |H_3:C_{H_3}(g)|_{p'} \leq q^{\frac{3m}2}c_m.$$
Putting these results together we conclude that
$$|H_1:C_{H_1}(g)|_{p'}\cdot |H_1:C_{H_1}(g)|\big|_{p'} \leq q^{\frac{3tm}2}(c_m)^t.$$
Now $|H/(H\cap N)|<2^{t-1}$ and so, by Lemma \ref{l: cn},
$$|H:C_H(g)|_{p'}\cdot|H:C_H(g)|\big|_{p'} \leq q^{\frac{3tm}2}(c_m)^t2^{t-1}<q^{\frac{3n}2}c_n$$
as required.
\end{proof}

\subsection{Case 2}

We are now ready to prove Lemma A  for $H$ satisfying Case 2 of Lemma \ref{l: relevant groups}, i.e. for $H\leq M\in\curly_i(\GL_n(q))$, for $i=1,\dots, 5$. (See p.\pageref{la} for the statement of Lemma A.) If $M\in\curly_5(\GL_n(q))$ then $M=\GL_n(q_0)Z(\GL_n(q))$ and it is enough to prove the result for $H\cap \GL_n(q_0)\leq \GL_n(q_0)$. We exclude this case and, for the rest, suppose that Lemma A holds for $m<n$, i.e. we proceed under the following hypothesis:
\begin{ihypothesis}
Let $H_1$ be an even order subgroup of $\GL_m(q)$ with $m<n$. Then $H_1$ contains an involution $g$ such that
$|H_1:C_{H_1}(g)|_{p'}\leq q^{m-1}+\dots+q+1$.
\end{ihypothesis}

\begin{lemma}
If $H$ has even order and $H\leq M\in\curly_i(\GL_n(q))$ with $i=1,3$, or $4$, then $H$ contains an involution $g$ such that
\[|H:C_H(g)|_{p'}\leq q^{n-1}+\dots+q+1.\]
\end{lemma}
\begin{proof}
If $H\leq M\in\curly_1(\GL_n(q))$ then $H\leq Q:(\GL_m(q)\times \GL_{n-m}(q))$ where $Q=O_p(M)$ and $m>1$. Set $L=M/Q$ and observe that, for any $g\in H$,
$$|H:C_H(g)|_{p'} = |H_L: C_{H_L}(g_L)|_{p'}$$
where $H_L$ is a subgroup of $L$ and $g_L$ is an involution in $H_L$. Now suppose, without loss of generality, that $m\leq n-m$, and set $N=H_L\cap \GL_m(q)$. 

Suppose that $N$ has odd order. Lemma \ref{L: oddnormal} implies that
$$|H_L: C_{H_L}(g_L)| = |N:C_N(g)| \times |H_L/N:C_{H_L/N}(g_LN)|.$$
Now $H/N$ is isomorphic to a subgroup of $\GL_{n-m}(q)$ and induction implies that we can choose $g_L$ such that 
$$|H/N:C_{H/N}(g_LN)|_{p'} \leq q^{n-m-1}+\dots+q+1.$$
Then $g_L$ normalizes $N$ and $g_L$ either centralizes $N$ or else $\langle g_L, N\rangle$ is isomorphic to a subgroup of $\GL_m(q)$. Either way, since $N$ has odd order, $\langle g_L, N\rangle$ has at most one conjugacy class of involutions, and induction implies that
$$|H_L:C_{H_L}(g_L)|_{p'} \leq q^{m-1}+\dots+q+1,$$
as required. If, on the other hand, $N$ has even order then Lemma \ref{L: sylowtwos} implies that, for $g_L\in N$,
$$|H_L: C_{H_L}(g_L)|\leq |N:C_{N}(g_L)|\times \frac{|g_L^{H_L}\cap P|}{|g_L^N\cap P|}$$
where $P$ is a Sylow $2$-subgroup of $N$. By induction we can choose $g_L$ so that
$$|N:C_{N}(g_L)|_{p'} \leq q^{m-1}+\dots+q+1.$$
Furthermore Lemma \ref{l: conj in sylow} implies that 
$$\frac{|g_L^{H_L}\cap P|}{|g_L^N\cap P|}<3.04^m<q^{m-1}+\dots+q+1.$$
Since $m\leq n-m$ the result follows.

If $H\leq M\in\curly_3(\GL_n(q))$ then $H\leq M=\GL_m(q^r).r$ where $n=mr$ and $r$ is prime. Let $N=\GL_m(q^r)$ be normal in $M$ and split into two cases. Suppose first that $|H\cap N|$ is even. Then induction implies that $H\cap N$ contains an involution $g$ such that
$$|H:C_H(g)|_{p'}\leq ((q^r)^{m-1}+\dots+q+1)r< q^{mr-1}+\dots+q+1$$
as required. Suppose on the other hand that $|H\cap N|$ is odd. Then $r=2$ and Lemma \ref{L: second gap} gives the result.

If $H\leq M\in\curly_4(\GL_n(q))$ then $M\cong \GL_{n_1}(q)\circ GL_{n_2}(q)$ where $n=n_1n_2$ and $1<n_1<n_2$. Proceed similarly to the case $\curly_1$, observing that $M$ has normal subgroups isomorphic to both $\GL_{n_1}(q)$ and $\GL_{n_2}(q)$; thus we define $N=H\cap \GL_{n_1}(q)$. Suppose first that $N$ has even order and take $g$ to be an involution in $N$. Then Lemma \ref{L: sylowtwos} implies that
$$|H:C_H(g)| = |N:C_N(g)|\frac{|g^H\cap P|}{|g^N\cap P|}$$
where $P$ is a Sylow $2$-subgroup of $N$. By induction we can choose $g$ so that 
$$|N:C_N(g)|_{p'}\leq q^{n_1-1}+\cdots+q+1.$$
Furthermore Lemma \ref{l: conj in sylow} implies that 
$$\frac{|g^H\cap P|}{|g^N\cap P|}<3.04^{n_1}< q^{n_1-1}+\dots+q+1,$$
and the result follows. 

Suppose, on the other hand, that $N$ has odd order. In particular this means that $H\cap Z(M)$ has order $z$, an odd number. Now Lemma \ref{L: oddnormal} implies that, for $g$ an involution in $H$,
$$|H: C_{H}(g)| = |N:C_N(g)| \times |H/N:C_{H/N}(gN)|.$$
Observe that $H/N$ is a subgroup of $\PGL_{n_2}(q)$. Let $H_X$ be the subgroup of $\GL_{n_1}(q)\times \GL_{n_2}(q)$ such that $H_X\cap Z(GL_{n_1}(q)\times \GL_{n_2}(q))\cong z\times z$, and such that $\pi(H_X)=H$ where $\pi$ is the natural map,
$$\pi:\GL_{n_1}(q)\times \GL_{n_2}(q)\to \GL_{n_1}(q)\circ \GL_{n_2}(q).$$
Let $H_2$ be the image of $H_X$ under the natural projection $\GL_{n_1}(q)\times \GL_{n_2}(q)\to GL_{n_2}(q)$. Observe that $H_2/(H_2\cap Z(\GL_{n_2}(q))\cong H/N$. By induction there exists an involution $g_2\in H_2$ such that 
$$|H_2: C_{H_2}(g_2)|_{p'}\leq q^{n_2-1}+\dots+q+1.$$
Let $g_X$ be an involution and preimage of $g_2$ in $H_X$; define $g=\pi(g_X)$, an involution in $H$. Observe that
$$|H/N:C_{H/N}(gN)| = |H_2: C_{H_2}(g_2)|_{p'}\leq q^{n_2-1}+\dots+q+1.$$
Furthermore either $g$ centralizes $N$ or elese $\langle g, N\rangle$ is isomorphic to a subgroup of $\GL_{n_1}(q)$. Either way, induction implies that 
$$|N:C_N(g)|_{p'}\leq q^{n_1-1}+\cdots+q+1,$$
and the result follows.
\end{proof}

\subsubsection{$\curly_2(\GL_n(q))$}

The final case to consider is when $H\leq M\in\curly_2(\GL_n(q))$. Then $M=\GL_m(q)\wr S_t, n=mt, t\geq2.$ Write 
$$N=\underbrace{\GL_m(q)\times\dots\times \GL_m(q)}_t.$$ 

We assume that $M$ acts primitively on the associated decomposition of the associated vector space $V$ into $t$ subspaces of dimension $m$, since otherwise $H$ lies in a maximal subgroup $M_1=\GL_{m_1}(q)\wr S_{t_1}$, $M_1\in\curly_2(\GL_n(q))$ with $t_1<t$ such that $H$ acts primitively on this decomposition.

We need a result of Praeger and Saxl \cite{ps}.

\begin{lemma}\label{l: cheryljan}
 Let $G$ be a primitive subgroup of the symmetric group $S_n$. Then $G$ contains $A_n$ or $|G|\leq 4^n$.
\end{lemma}

An immediate corollary is the following:

\begin{corollary}\label{c: cheryljan}
Let $G$ be a primitive subgroup of the symmetric group $S_n$. If $n\geq 5$, then $G$ contains an involution $g$ such that $|G:C_G(g)|\leq 5^{n-1}$.
\end{corollary}
\begin{proof}
Lemma~\ref{l: cheryljan} implies that either $|H/H\cap N|\leq 4^n$ (which yields the result for $n\geq 8$) or $|H/H\cap N|$ contains the alternating group $A_n$.

If $H/H\cap N$ contains the alternating group $A_n$, then $H/H\cap N$ contains an involution $g$ such that
\[
  |H/(H\cap N): C_{H/(H\cap N)}(g(H\cap N))|<n^4.
 \]
Again this yields the result for $n\geq 8$. If $n<8$, then $|S_n|=n!<5^{n-1}$ and we are done. 
\end{proof}

\begin{lemma}
Suppose that $H\cap N$ has odd order, $H$ has even order, and $t<n$. Then $H$ contains an involution $g$ such that
$|H:C_H(g)|_{p'}\leq q^{n-1}+\dots+q+1$.
\end{lemma}
\begin{proof}
By Corollary~\ref{c: oddgln},
$$|H\cap N|_{p'}<(q^{\frac{3m-2}{2}})^t.$$
Write $N=N_1\times \cdots N_t$ where $N_i\cong \GL_m(q), i=1,\dots, t$. Then, for $g$ an involution in $H$, we can choose an ordering of the $N_i$ so that $g$ either normalizes $N_i$ or swaps $N_i$ with $N_{i\pm 1}$. Let $H_i$ be the projection of $H$ into $N_i$; if $g$ normalizes $N_i$ then write $g_i$ for the element of $N_i$ induced by the action of $g$. Since $N$ is odd, $\langle g_i, H_i\rangle$ contains at most one conjugacy class of involutions and induction implies that
$$|H_i: C_{H_i}(g_i)|_{p'}\leq q^{m-1}+\cdots +q+1.$$
If $g$ swaps $N_i$ and $N_{i+1}$ then write $H_{i,i+1}$ for the projection of $H$ into $N_i\times N_{i+1}$ (observe that the projection of $H$ onto $N_i$ is isomorphic to the projection of $H$ onto $N_{i+1}$); write $g_{i,i+1}$ for the element of $N_i\times N_{i+1}$ induced by the action of $g$. Then Corollary~\ref{c: oddgln} implies that
$$|H_{i,i+1}: C_{H_{i,i+1}}(g_{i,i+1})|\leq \sqrt{|H_{i,i+1}|}\leq q^{\frac{3m-2}4}.$$
Now write $g$ as a product of $k$ disjoint transpositions in its action upon the set $\{N_1, \dots, N_t\}$. Then
$$|H\cap N:C_{H\cap N}(g)|_{p'}\leq (q^{\frac{3m-2}2})^k(q^{m-1}+\cdots +q+1)^{t-2k}\leq \frac{q^n-1}{(q-1)^t}.$$
Now Lemma \ref{L: oddnormal} implies that
$$|H:C_H(g)|=|H\cap N:C_{H\cap N}(g)|\times |H/(H\cap N): C_{H/(H\cap N)}(g(H\cap N))|,$$
hence it is enough to prove that $H/(H\cap N)$ contains an involution $h$ such that 
$$|H/(H\cap N): C_{H/(H\cap N)}(h)|_{p'}\leq (q-1)^{t-1}.$$
Corollary~\ref{c: cheryljan} yields the result.
\end{proof}

We next consider the possibility that $H\cap N$ has odd order and $t=n$.

\begin{lemma}
Suppose that $H\cap N$ has odd order and $t=n$. Then $H$ contains an involution $g$ such that
$|H:C_H(g)|_{p'}\leq q^{n-1}+\dots+q+1$.
\end{lemma}
\begin{proof}
Observe that $N=(q-1)^n$. First consider $C_{S}(g)$ where $S$ is the projection of $H$ onto a particular $q-1$ which is fixed by $g$. Clearly $|S: C_{S}(g)|_{p'}=1.$ Alternatively if $(q-1)\times (q-1)$ are swapped by $g$, and $S$ is the projection of $H$ onto $(q-1)\times(q-1)$ then it is clear that $|S: C_{S}(g)|_{p'}< \frac{q-1}{2}$. Thus, in all cases,
$$|N:C_N(g)|_{p'}<(\frac{q-1}{2})^{\frac{n}{2}}.$$
Lemma \ref{L: oddnormal} implies that it is sufficient to prove that $H/N$ contains an involution $h$ such that 
\[|H/N:C_{H/N}(h)|\leq 2^{n/2}(q-1)^{n/2-1}.\]
Corollary~\ref{c: cheryljan} implies that $H/N$ contains an involutions $h$ such that
\[|H/N:C_{H/N}(h)|\leq 5^{n-1}\]
and this yields the result provided $q\geq 43$ and $n\geq 6$.

If $3\leq n\leq 5$, then $|H/N|_{}= 3$ and the result follows. If $n=2$, then $|H/N|_{}= 1$ and the result follows.

If $q<43$ then $q=7,13,19,31$ or $37$. In all cases $|q-1|_{p'}=3$ and hence $|N:C_N(g)|_{p'}\leq3$. Thus, by Lemma \ref{L: oddnormal}, 
$$|H:C_H(g)|_{p'}\leq 3\times 5^{n-1}\leq q^{n-1}+\dots+q+1.$$
This yields the result for $n\geq 6$ and we are done.
\end{proof}

The final subcase is when $H\cap N$ has even order. We need a couple of classical theorems, the first is Glauberman's $Z^*$-theorem \cite{glaub}.

\begin{theorem}\label{t: zstar}
 Let $x$ be an element of $S$, a Sylow $2$-subgroup of a finite group $K$. Let $O(K)$ be the largest normal odd-orber subgroup of $K$ and let $Z^*(K)$ be the preimage of $Z(K/O(K))$ in $K$. Then $x\not \in Z^*(G)$ if and only if there exists $y\in C_s(x)\backslash\{x\}$ such that $y$ is conjugate to $x$ in $G$. 
\end{theorem}

The second is (an easy consequence of) Alperin's fusion theorem (see \cite[p.244]{gorenstein3}).

\begin{theorem}\label{t: alperin}
Let $g$ and $h$ be elements of $S$, a Sylow $t$-subgroup of a finite group $K$. Suppose that $g^x=h$ for some $x\in K$. Then there exist elements $x_i$ and Sylow $t$-subgroups $Q_i$ of $G$, $1\leq i\leq n$, and an element $y\in N_G(S)$ such that $x=x_1\cdots x_n y$, $g\in P\cap Q_1$ and, for $i=1,\dots, n$, $x_i\in N_G(P\cap Q_i$ and $g^{x_1x_2\cdots x_i}\in P\cap Q_{i+1}$.
\end{theorem}

\begin{lemma}\label{l: el}
 Suppose that $E$ is an elementary abelian $2$-subgroup of $\GL_n(q)$ with $q$ odd. Then $|E|\leq 2^n$.
\end{lemma}
\begin{proof}
 A commuting set of semisimple elements is the subset of a torus of $\GL_n(q)$. Since a torus is a direct product of at most $n$ cyclic groups, the result follows.
\end{proof}

We are ready to deal with the final subcase.

\begin{lemma}
Suppose that $H\leq \GL_m(q)\wr S_t$ where $n=mt$. Write 
$$N=\underbrace{\GL_m(q)\times\dots\times \GL_m(q)}_t.$$
Suppose that $H\cap N$ has even order. Then $H$ contains an involution $g$ such that
$$|H:C_H(g)|_{p'}\leq q^{n-1}+\dots+q+1.$$
\end{lemma}
\begin{proof}

Suppose first that $m=1$. Then $N$ contains at most $2^n$ involutions and so 
$$|H:C_H(g)|_{p'}<2^n\leq q^{n-1}+\dots+q+1,$$
as required. Assume now that $m>1$.

Let $V$ be the associated vector space for $\GL_n(q)$. Write $N=N_1\times \dots \times N_t$ with $N_i\cong \GL_m(q)$ for $i=1,\dots,t$. Our analysis (and our notation) mirrors the set up in Lemma \ref{L: invcentralizer}. Write $L_i$ for the projection of $H\cap N$ onto $N_i\times\dots\times N_t$ and write $\psi_i:L_i\to L_{i+1}$ for the natural projection. Let $T_i$ be the kernel of $\psi_i$ for $1\leq i<t$; define $T_t:= L_t$. Suppose that $|T_i|$ is odd for $i<k\leq t$ and $|T_i|$ is even for $i=k$.

Observe that $T_k\leq \GL_m(q)$. Then, by induction, take an involution $g_k\in T_k$ such that 
$$|T_k:C_{T_k}(g_k)|_{p'}\leq q^{m-1}+\cdots+q+1.$$
Let $g\in G$ be an involution and pre-image of $g_k$ in $H\cap N$ (such a pre-image must exist since $|T_i|$ is odd for $i<k$. Write $g_i$ for the image of $g$ under the projection of $H\cap N$ onto $L_i$. We proceed in stages:

\begin{enumerate}
\item {\bf An upper bound for} $|H\cap N:C_{H\cap N}(g)|_{p'}$.  Since $|T_i|$ is odd for $i<k$, $\langle T_i, g_i\rangle$ contains at most one conjugacy class of involutions, and induction implies that 
$$|T_i: C_{T_i}(g_i)|_{p'}\leq q^{m-1}+\dots+q+1,$$
and Lemma \ref{l: conj in sylow} implies that $|(g_k)^{L_k}\cap P_k|<2^{m+\frac{m}2-1}$ where $P_k$ is a Sylow $2$-subgroup of $T_k$. Furthermore Lemma~\ref{L: invcentralizer} implies that 
\[|H\cap N:C_{H\cap N}(g)|_{p'}=\prod_{i=1}^k|T_i:C_{T_i}(g_i)|_{p'}\frac{|(g_k)^{L_k}\cap P_k|_{p'}}{|(g_k)^{T_k}\cap P_k|_{p'}}.\]
(Note that if $k=t$ the final fraction is equal to $1$.) Now combining these observations we obtain that 
\begin{equation*}
 \begin{aligned}
|H\cap N:C_{H\cap N}(g)|_{p'}&=\prod_{i=1}^k|T_i:C_{T_i}(g_i)|_{p'}\frac{|(g_k)^{L_k}\cap P_k|_{p'}}{|(g_k)^{T_k}\cap P_k|_{p'}} \\
&\leq (q^{m-1}+\cdots+q+1)^k \frac{|(g_k)^{L_k}\cap P_k|_{p'}}{|(g_k)^{T_k}\cap P_k|_{p'}} \\
& \leq 
\left\{\begin{array}{ll}
(q^{m-1}+\cdots+q+1)^k 2^{m+\frac{m}2-1}, & k<t \\
(q^{m-1}+\cdots+q+1)^t, & k = t
\end{array}\right. \\
& \leq \frac{(q^m-1)^t}{(q-1)^t} \\
& < \frac{q^{n-1}+\cdots+q+1}{(q-1)^{t-1}}.
\end{aligned}
\end{equation*}

Notice that if $|H:C_H(g)| = |H\cap N:C_{H\cap N}(g)|$ then the result follows immediately; thus we assume that this is not the case, i.e. that $H$ fuses the conjugacy class containing $g$ to another.

\item {\bf A lower bound for} $t$. It is clear that
\begin{equation*}
\begin{aligned}
|H:C_H(g)|_{p'} 
&\leq|H\cap N:C_{H\cap N}(g)|_{p'}|H/(H\cap N)|_{p'} \\
&<\frac{q^{n-1}+\cdots+q+1}{(q-1)^{t-1}} |t!|_{p'}
\end{aligned}
\end{equation*}
Thus the result holds provided $|t!|_{p'}\leq6^{t-1}$. So assume that $|t!|_{p'}>6^{t-1}$; in particular assume that $t>20$.


\item {\bf A lower bound for} $k$. Let $P$ be a Sylow $2$-subgroup of $H\cap N$
Lemma \ref{L: sylowtwos} implies that, for $g\in P$,
\begin{equation}\label{m}
|H:C_H(g)|_{p'} =|H\cap N:C_{H\cap N}(g)|_{p'}\frac{|g^H\cap P|_{p'}}{|g^{H\cap N}\cap P|_{p'}}.
\end{equation}
Now let $\phi:N\to L_k$ be the natural projection map; then $\phi(P)\cong P$ is a Sylow $2$-subgroup of $L_k$. Let $P_k=\phi(P)\cap T_k$ and choose $g$ such that $\phi(g)\in P_k$. Then Lemma~\ref{L: invcentralizer} implies that
\begin{equation}\label{n}
 |H\cap N:C_{H\cap N}(g)|_{p'}=\prod_{i=1}^k|T_i:C_{T_i}(g_i)|_{p'}\frac{|(g_k)^{L_k}\cap P_k|_{p'}}{|(g_k)^{T_k}\cap P_k|_{p'}}.
\end{equation}
Let $P'\leq P$ be such that $\phi(P')=P_k$; let $T$ be the full pre-image of $T_k$ in $H\cap N$. 
Then Lemma \ref{l: sylow odd order} implies that
\begin{equation}\label{o}
 \begin{aligned}
  |g^{H\cap N}\cap P| &= |(g_k)^{L_k}\cap P_k| \\
  |g^{T}\cap P| &= |(g_k)^{T_k}\cap P_k| \\
\end{aligned}
\end{equation}
Combining \eqref{m}, \eqref{n} and \eqref{o}, we obtain that
\begin{equation*}
|H:C_H(g)|_{p'}=\prod_{i=1}^k|T_i:C_{T_i}(g_i)|_{p'}\frac{|g^{H}\cap P|_{p'}}{|g^{T}\cap P|_{p'}}.
\end{equation*}
Since $T_k$ has odd order, $\langle T_k, g_k\rangle$ has at most one conjugacy class of involutions. Then Lemma \ref{l: conj in sylow} and induction imply that 
\begin{equation*}
 \begin{aligned}
|H:C_H(g)|_{p'}&<(q^{m-1}+\dots+q+1)^k (3.04^m+1)^{t-k+1}
\end{aligned}
\end{equation*}
This yields the result unless $k=t$ or $(q,k)=(7,t-1)$. We assume that one of these holds from here on and we note that if $k=t$, then $T=H\cap N$. We assume, moreover, that when $q=7$ we have chosen an ordering for $N_1,\dots, N_k$ so that $k$ is minimal.

\item {\bf A lower bound for} $m$. From the previous paragraph observe that
\begin{equation*}
\begin{aligned}
|H:C_H(g)|_{p'}&<(q^{m-1}+\dots+q+1)^k (3.04^m+1)^{t-k+1} \\
&\leq (q^{m-1}+\dots+q+1)^t(3.04^m+1) \\
&=(q^m-1)^t\frac{3.04^m+1}{(q-1)^t}\\
&<(q^{n-1}+\dots+q+1)\frac{3.04^m+1}{(q-1)^{t-1}}.
 \end{aligned}
\end{equation*}
If $3.04^m+1\leq (q-1)^{t-1}$ then the result follows. Thus we assume that $3.04^m+1> (q-1)^{t-1}$; in particular $m\geq t>20$ and $n\geq 400$.

\item{\bf No conjugates of $g_k$ in $P_k$.} By assumption $g_k\in P_k$ is chosen so that $|T_k:C_{T_k}(g_k)|_{p'}$ is minimal. Suppose that there were no other $T_k$-conjugates of $g_k$ in $P_k$. Then Theorem~\ref{t: zstar} implies that $g_k \in Z^*(T_k)$ and so $g\in Z^*(T)$. In particular 
\[|H:C_H(g)|_{p'}\leq |T:C_T(g)|_{p'} \cdot |Z_2(T/O(T))|\]
where $Z_2(T/O(T))$ is the unique maximal elementary abelian 2-subgroup of $Z(T/O(T))$. (Recall that $O(T)$ is the largest odd-order normal subgroup of $T$.)
Now, since \[
 |T:C_T(g)|_{p'}=\prod_{i=1}^k|T_i:C_{T_i}(g_i)|_{p'}\leq \frac{(q^m-1)^t}{(q-1)^t}.
\]
the result follows unless $|Z_2(T/O(T))|\geq (q-1)^{t-1}$. Thus assume that this is the case and let $W$ be the pre-image of $Z_2(T/O(T))$ in $H\cap N$ and let $P'$ be a Sylow $2$-subgroup of $W$. Clearly $W=O(T) P'$ and $P'$ is elementary abelian. Take $g,h$ to be distinct elements of $P'\backslash\{1\}$ and observe that Lemma~\ref{L: oddnormal} implies that $|T:C_T(g)| = |O(T):C_{O(T)}(g)|$ (similarly for $h$).


Observe that $\langle O(T), g, h\rangle$ is isomorphic to a subgroup of $\GL_{mk}(q)$ and that $\langle g,h\rangle$ is a Klein 4-group. We apply Lemma~\ref{L: klein} to the group $\langle O(T), g,h\rangle$ and (by relabelling $g$ and $h$ if necessary) conclude that
\[
|T:C_T(g)|_{p'}= |O(T):C_{O(T)}(g)|_{p'} \leq \sqrt{q^{\frac{3mk}{2}}c_{mk}} \leq 3^{\frac{mk}{6}}\cdot q^{\frac{3mk}4}.
\]
If $t=k$, then $T=H\cap N$ and $P$ is is isomorphic to a $2$-subgroup of $\GL_m(q)$. Now Lemma~\ref{l: conj in sylow} and the fact that $m\geq t\geq 20$ imply that
\[
 |H:C_H(g)|_{p'}\leq 3^{\frac{n}{6}}\cdot q^{\frac{3n}4}\cdot 3.04^m<q^{\frac{29n}{30}}<q^{n-1}
\]
and we are done.

If $t<k$, then $q=7$ and $k=t-1$. Now observe that $W$ is normal in $H\cap N$ and that $|W:C_W(g)| = |T:C_T(g)|$. Applying Lemma~\ref{L: sylowtwos} we have
\[
 |H\cap N:C_{H\cap N}(g)|_{p'} = |T:C_T(g)|_{p'} \frac{|g^{H\cap N} \cap P'|}{|g^W\cap P'|}.
\]
Now $P'$ is isomorphic to an elementary abelian 2-subgroup of $\GL_m(q)$, thus Lemma~\ref{l: el} implies that $|P'| \leq 2^m$ and we have
\[
 |H\cap N:C_{H\cap N}(g)|_{p'}< \sqrt{q^{\frac{3m(t-1)}{2}}c_{m(t-1)}}\cdot 2^m\]
Since $P$ is isomorphic to a $2$-subgroup of $\GL_m(q)\times\GL_m(q)$, Lemma~\ref{l: conj in sylow} implies that
 \[
  \begin{aligned}
 |H:C_H(g)|_{p'} &<   \sqrt{q^{\frac{3m(t-1)}{2}}c_{m(t-1)}}\cdot 2^m\cdot 3.04^{2m} \\
&\leq 3^{\frac{n-m}{6}}q^{\frac{3n-3m}{4}}\cdot 2^m \cdot 3.04^{2m} \\
&\leq 3^{\frac{n}{6}-\frac{n}{120}}q^{\frac{3n}{4}-\frac{3n}{80}} \cdot 12.16^{\frac{n}{10}} \\
&< q^{0.94n}<q^{n-1}
 \end{aligned}
 \]
and the result follows.

\item{\bf Conjugates of $g_k$ in $P_k$}. 
Suppose, instead, that $P_k$ contains other $T_k$-conjugates of $g_k$ in $P_k$. Then $P$ contains other $T$-conjugates of $g$ in $P$. Now Theorem~\ref{t: alperin} implies that there exists a $T$-conjugate $h\in P_k$ such that $\langle g,h\rangle$ is a Klein $4$-subgroup of $P$.

Let $V$ be the normal subgroup of $T$ that is the pre-image of $T_{k-1}$. Observe that $\langle V, g, h\rangle$ is isomorphic to an odd order subgroup of $\GL_{m(k-1)}(q)$. We apply Lemma~\ref{L: klein} to the group $\langle V, g,h\rangle$ and (by relabelling $g$ and $h$ if necessary) conclude that
\[
 |V:C_V(g)|_{p'} \leq \sqrt{q^{\frac{3m(k-1)}{2}}c_{m(k-1)}} \leq 3^{\frac{m(k-1)}{6}}\cdot q^{\frac{3m(k-1)}4}. 
\]
If $t=k$, then $T=H\cap N$ and Lemma~\ref{L: oddnormal} implies that
\[\begin{aligned}
 |H\cap N:C_{H\cap N}(g)|_{p'} &= |V:C_V(g)|_{p'} |T_k:C_{T_k}(g_k)|_{p'}\\ &\leq 3^{\frac{m(t-1)}{6}}\cdot q^{\frac{3m(t-1)}4} (q^{m-1}+\cdots+q+1).\end{aligned}
\]
Then, since $P$ is isomorphic to a subgroup of $\GL_m(q)$, Lemma~\ref{L: sylowtwos} and \ref{l: conj in sylow} imply that
\[
 \begin{aligned}
|H:C_H(g)|_{p'} &\leq  3^{\frac{m(t-1)}{6}}\cdot q^{\frac{3m(t-1)}4} (q^{m-1}+\cdots+q+1)3.04^m \\
&< 3^{\frac{n}{6}-\frac{n}{120}}q^{\frac{3n}{4}+\frac{n}{80}}3.04^{\frac{n}{20}} \\
&< q^{\frac{39n}{40}} <q^{n-1}
 \end{aligned}
\]
and we are done.

If $t=k-1$, then
\[\begin{aligned}
 |T:C_T(g)|_{p'} &\leq \sqrt{q^{\frac{3m(k-1)}{2}}c_{m(k-1)}}(q^{m-1}+\cdots+q+1) \\ &\leq 3^{\frac{m(t-2)}{6}}\cdot q^{\frac{3m(t-2)}4}(q^{m-1}+\cdots+q+1). \end{aligned}
\]
Now, since $P\cap T$ is isomorphic to a subgroup of $\GL_m(q)$, Lemma~\ref{L: sylowtwos} and \ref{l: conj in sylow} imply that
\[
 |H\cap N:C_{H\cap N}(g)|_{p'} \leq 3^{\frac{m(t-2)}{6}}\cdot q^{\frac{3m(t-2)}4}(q^{m-1}+\cdots+q+1)3.04^m.
\]
Similarly, since $P$ is isomorphic to a subgroup of $\GL_m(q)\times\GL_m(q)$, Lemma~\ref{L: sylowtwos} and \ref{l: conj in sylow} imply that
\[
\begin{aligned}
 |H:C_{H}(g)|_{p'} &\leq 3^{\frac{m(t-2)}{6}}\cdot q^{\frac{3m(t-2)}4}(q^{m-1}+\cdots+q+1)3.04^m3.04^{2m} \\
&\leq 3^{\frac{n}{6}-\frac{n}{60}}q^{\frac{3n}{4}-\frac{3n}{40}}q^{\frac{n}{20}}3.04^{\frac{3n}{20}} \\
&<q^{0.98n}<q^{n-1}
\end{aligned}
 \]
 and the result is proved.
\end{enumerate}

\end{proof}

This completes our proof of Lemma A, and thereby proves Theorem A.

\section{Theorem B and its corollaries}\label{S: quaternionsylowtwos}

Our first task is to prove Theorem B. We begin by recalling the statement of Hypothesis 1.

\begin{hypothesisone}
Suppose that a group $G$ acts transitively upon the points of a non-Desarguesian projective plane $\spaceP$ of order $x>4$. If $G$ contains any involutions then each fixes a Baer subplane; in particular they fix $u^2+u+1$ points where $x=u^2, u>2$. Furthermore in this case any non-trivial element of $G$ fixing at least $u^2$ points fixes either $u^2+u+1, u^2+1$ or $u^2+2$ points of $\spaceP$.
\end{hypothesisone}

In this section we operate under Hypothesis 1, which we recall represents the conditions under which a group may act transitively on the points of a non-Desarguesian projective plane (see p.\pageref{h:basic}).

Our proof of Theorem B will require a background result and a lemma. The background result follows from the comments in \cite[p377]{gorenstein3} and the results in \cite{gw1}.

\begin{proposition}\label{p: quart}
If a group $H$ has generalized quaternion Sylow $2$-subgroups then $H/O(H)$ is isomorphic to one of the following groups:
\begin{enumerate}
\item $S$, i.e. $H$ has a normal $2$-complement.
\item $2.A_7$, a double cover of $A_7$.
\item $\SL_2(q).D$, where $q$ is odd and $D$ is cyclic.
\end{enumerate}
\end{proposition}

\begin{lemmab}
Suppose that $G$ has generalized quaternion Sylow $2$-subgroups and is insoluble. Then $G$ has a subgroup of index at most $2$ equal to $O(G)\rtimes K$ where $K$ is a group isomorphic to $\SL_2(5)$. Furthermore if $N_t$ is a Sylow $t$-subgroup of $F(G)$, for some prime $t$, then one of the following holds:
\begin{enumerate}
 \item[(i)] $t$ divides $u^2+u+1$;
 \item[(ii)] $t$ divides $u^2-u+1$, $N_G(N_t)$ contains a subgroup $H$ such that $H\cong \SL_2(5)$, $H$ fixes a point  of $\spaceP$, $N_t\rtimes H$ is a Frobenius group, and $N_t$ is abelian.
\end{enumerate}
Furthemore there exists a prime $t$ dividing both $|F(G)|$ and $u^2-u+1$.
\end{lemmab}

We defer the proof of Lemma B until the next subsection; first we demonstrate that Lemma B implies Theorem B. We need a little notation which will hold throughout this section: for a group $H\leq G$, define $\overline{H}=H/(H\cap O(G))$. 

\begin{proof}[Proof that Lemma B implies Theorem B]
Let $S$ be a Sylow $2$-subgroup of $G$; Theorem A implies that $S$ is cyclic or generalized quaternion. Note that, since $S$ contains exactly one involution, the Frattini argument implies that $G=O(G)C_G(g)$ for any involution $g\in G$.

If $S$ is cyclic then \cite[39.2)]{aschbacher3} implies that $\overline{G}\cong S$, i.e. $G$ has a {\it normal $2$-complement} and Theorem B holds.

If $S$ is generalized quaternion then Proposition~\ref{p: quart} gives the structure of $\overline{G}$. Clearly if (1) holds, then Theorem B holds. Similarly if (3) holds with $q=3$, then Theorem B holds. Thus, to prove Theorem B we must deal with the remaining cases. Equivalently it is sufficient to assume that $G$ has generalized quaternion Sylow $2$-subgroups and is insoluble. Now Lemma B applies and Theorem B follows immediately.
\end{proof}

\subsection{A proof of Lemma B}

Throughout this subsection we operate under the supposition of Lemma B, as well as under Hypothesis~\ref{h:basic}. In particular this means that $F^*(\overline{G})$ is a quasisimple group with centre of order $2$. We start with a background result.

\begin{lemma}\label{l: asch}\cite[(5.21)]{aschbacher3}
Let a group $J$ be transitive on a set $X$, $x\in X$, $H$ the stabilizer in $J$ of $x$, and $K\leq H$. Then $N_J(K)$ is transitive on $\Fix(K)$ if and only if $K^J\cap H=K^H$.
\end{lemma}

\begin{lemma}
For $g$ an involution in $G$, $C_G(g)$ acts transitively on the set of points of $\Fix(g)$, a Baer subplane of $\spaceP$.
\end{lemma}
\begin{proof}
This is immediate from Lemma \ref{l: asch}. Take $G$ to be the group $J$, $X$ to be the set of points of $\spaceP$; thus $H=\Galph,$ the stabilizer of a point $\alpha$. Define $K$ to be a subgroup of $\Galph$ of order $2$, say $K=\{1,g\}$. Now $g$ is the unique involution in a particular Sylow $2$-subgroup of $\Galph$ and so $K^G\cap \Galph=K^{\Galph}$. Hence $N_G(K)=C_G(g)$ is transitive on $\Fix(K)=\Fix(g)$, a Baer subplane of $\spaceP$. 
\end{proof}

For a fixed involution $g\in G$, define the subgroup
$$T_g=\{h\in C_G(g) | \Fix(h)=\Fix(g)\}.$$ 
Thus $T_g$ is the kernel of the action of $C_G(g)$ on $\Fix(g)$. 


\begin{lemma}\label{l: T}
For $g$ an involution in $G$, $\overline{T_g}$ contains a normal subgroup isomorphic to $F^*(\overline{G})$.
\end{lemma}
\begin{proof}
Suppose first that $\Fix(g)$ is Desarguesian. There are two possibilities for the structure of $C_G(g)/T_g$: either $C_G(g)/T_g$ is soluble or $C_G(g)/T_g$ has socle $\PSL_3(u)$ where $u>2$ is the order of $\Fix(g)$. Now, since $u>2$, $\PSL_3(u)$ has Sylow $2$-subgroups that are neither cyclic nor dihedral and so they cannot form a section of a quaternion group. Hence we conclude that $C_G(g)/T_g$ is soluble.

Now suppose that $\Fix(g)$ is not Desarguesian. Recall that the quotient of a quaternion group $Q$ by a normal subgroup containing the unique central involution of $Q$ must be either dihedral or cyclic. We conclude that $C_G(g)/T_g$ must have dihedral or cyclic Sylow $2$-subgroups. The former is not possible by Theorem A and the latter implies that $C_G(g)/T_g$ is soluble. Hence this conclusion holds in all cases.

Since $G=O(G)C_G(g)$, we know that $\overline{C_G(g)}$ must contain a normal subgroup isomorphic to $F^*(\overline{G})$. Since $C_G(g)/T_g$ is soluble, we conclude that $\overline{T_g}$ must contain a normal subgroup isomorphic to $F^*(\overline{G})$.
\end{proof}

Now we will need a background result which is originally due to Zassenhaus \cite{zassenhaus}; it is discussed fully in \cite{passman}.

\begin{theorem}\label{t: frobenius}
Let $F$ be an insoluble Frobenius complement. Then $F$ has a normal subgroup $F_0$ of index at most $2$ such that $F_0=SL_2(5)\times M$ with $M$ a group of order prime to $2,3$ and $5$.
\end{theorem}
We are now in a position to prove Lemma B and hence Theorem B:

\begin{proof}[Proof of Lemma B]
Lemma~\ref{L:firstly} implies that any prime dividing $|F(G)|$ divides 
\[u^4+u^2+1 = (u^2+u+1)(u^2-u+1).\] 
Hypothesis \ref{h:basic}, Proposition \ref{prop:implication} and Lemma A together imply that there is a prime $t$ dividing $|F(G)|$ that does not divide $u^2+u+1$. We conclude that $t$ divides $u^2-u+1$ and we write $N_t$ for the Sylow $t$-subgroup of $F(G)$.

Let $g$ be an involution in $G$. Since $N_t$ acts semi-regularly on the points of $\spaceP$ (see Lemma \ref{L:firstly}), we know that $T_g\cap N_t$ is trival; we also know that $T_g$ acts (by conjugation) on $N_t$.

Suppose that $T_g$ does not act semi-regularly on $N_t$. This implies that there exists $h\in T_g, h\neq 1$ such that $C_{N_t}(h)$ is non-trivial. Now $C_{N_t}(h)$ acts on the fixed set of $h$ and, since $h\in T_g$, we know that the fixed set of $h$ is a Baer subplane. We know, furthermore, that $N_t$ acts semi-regularly on the points of $\spaceP$. This implies that $t$ divides $u^2+u+1$ which is a contradiction. 

We conclude that $T_g$ acts semi-regularly on $N_t$, i.e. that $N_t\rtimes T_g$ is a Frobenius group with Frobenius complement $T_g$ and Frobenius kernel $N_t$.

We can now apply Theorem \ref{t: frobenius}; thus $T_g$ has a subgroup $T_{g,0}$ of index at most $2$ such that $T_{g,0}=H\times M$ where $H\cong SL_2(5)$ and $(|M|,60)=1$. Note that $H$ is the unique normal quasisimple subgroup of $T_g$. Observe that $N_t\rtimes H$ is a subgroup of $G$ that is also a Frobenius group and so, in particular, $N_t$ is abelian \cite[Theorem 6.3]{isaacs}. What is more $N_t\leq T_g$, thus $N_t$ fixes a point of $\spaceP$.

Now Lemma \ref{l: T} implies that $F^*(\overline{G})\cong SL_2(5)$; Proposition \ref{p: quart} implies that $\overline{G}$ is a cyclic extension of $\SL_2(5)$. Since $\rm{Aut}(SL_2(5)) = \PSL_2(5).2$, this extension must have size at most $2$; thus $\overline{G} \cong SL_2(5)$ or $\SL_2(5).2$.

Finally, since $T_g$ has a subgroup isomorphic to $\SL_2(5)$ it is clear that $G$ contains a split extension, $O(G):SL_2(5),$ and the result follows.
\end{proof}

\subsection{Corollary~\ref{c: odd order}}


\begin{proof}[Proof of Corollary~\ref{c: odd order}]
As we pointed out in the introduction, the statement is true for Desarguesian projective planes. Assume that $\spaceP$ is non-Desarguesian. Write $v$ for the number of points in $\spaceP$ and let $H=O^2(G)$, the largest normal subgroup of $G$ such that $|G/H|$ is a (possibly trivial) power of $2$. Since $v$ is odd, it is clear that $H$ acts transitively on the set of points of $\spaceP$.

Now apply Theorem B. If (1) applies, then $H=O(G)$, an odd-order group and we are done. If (2) applies, then $H=O(H).SL_2(3)$ and $O(H)$ lies inside an odd-order subgroup $M$ of index $8$ in $H$; then $M$ acts transitively on the set of points of $\spaceP$ and we are done. Finally if (3) applies, then Theorem B implies that $H=O(H)\rtimes K$ where $K\cong SL_2(5)$ and, moreover, $K$ fixes a point of $\spaceP$. Thus $O(H)$ acts transitively on the set of points of $\spaceP$ and we are done.
\end{proof}

\def\cprime{$'$}
\providecommand{\bysame}{\leavevmode\hbox to3em{\hrulefill}\thinspace}
\providecommand{\MR}{\relax\ifhmode\unskip\space\fi MR }
\providecommand{\MRhref}[2]{%
  \href{http://www.ams.org/mathscinet-getitem?mr=#1}{#2}
}
\providecommand{\href}[2]{#2}

\end{document}